\documentclass[12pt, a4paper]{amsart}
\usepackage{amsmath}
\usepackage{geometry,amsthm,graphics,tabularx,amssymb,shapepar}
\usepackage{amscd}
\usepackage[all]{xy}

\newcommand{\CB}{{\mathcal {B}}}

\newcommand{\CZ}{{\mathcal {Z}}}

\newcommand{\RN}{{\mathrm {N}}}

\newcommand{\Ad}{{\mathrm{Ad}}}

\newcommand{\GL}{{\mathrm{GL}}}

\newcommand{\Hom}{{\mathrm{Hom}}}

\newcommand{\Ind}{{\mathrm{Ind}}}

\newcommand{\Lie}{{\mathrm{Lie}}}

\newcommand{\SL}{{\mathrm{SL}}}

\newcommand{\SO}{{\mathrm{SO}}}

\newcommand{\SU}{{\mathrm{SU}}}

\newcommand{\sgn}{{\mathrm{sgn}}}
\newcommand{\Sp}{{\mathrm{Sp}}}

\newcommand{\cf}{\emph{cf.}~}

\newcommand{\con}{\textit{C}}

\newcommand{\od}{\operatorname{d}}

\newcommand{\oL}{\operatorname{L}}

\newcommand{\oH}{\operatorname{H}}
\newcommand{\oO}{\operatorname{O}}
\newcommand{\oS}{\operatorname{S}}
\newcommand{\oR}{\operatorname{R}}

\newcommand{\oU}{\operatorname{U}}

\newcommand{\g}{\mathfrak g}

\renewcommand{\k}{\mathfrak k}
\newcommand{\h}{\mathfrak h}

\newcommand{\q}{\mathfrak q}

\renewcommand{\b}{\mathfrak b}
\renewcommand{\c}{\mathfrak c}
\newcommand{\n}{\mathfrak n}

\renewcommand{\l}{\mathfrak l}

\newcommand{\m}{\mathfrak m}

\renewcommand{\sl}{\mathfrak s \mathfrak l}
\newcommand{\gl}{\mathfrak g \mathfrak l}

\newcommand{\Z}{\mathbb{Z}}
\newcommand{\C}{\mathbb{C}}
\newcommand{\R}{\mathbb R}

\newcommand{\K}{\mathbb{K}}

\newcommand{\abs}[1]{\lvert#1\rvert}

\newcommand{\la}{\langle}
\newcommand{\ra}{\rangle}

\newcommand{\be}{\begin {equation}}
\newcommand{\ee}{\end {equation}}
\newcommand{\bee}{\begin {equation*}}
\newcommand{\eee}{\end {equation*}}

\theoremstyle{Theorem}

\theoremstyle{Theorem}

\newtheorem{thm}{Theorem}[section]

\newtheorem{lemt}[thm]{Lemma}
\newtheorem{prpt}[thm]{Proposition}

\theoremstyle{Theorem}
\newtheorem{lem}{Lemma}[section]

\newtheorem{thml}[lem]{Theorem}

\theoremstyle{Theorem}

\theoremstyle{remark}
\newtheorem*{remark}{Remark}

\newtheorem*{remarks}{Remarks}

\theoremstyle{Definition}

\theoremstyle{remark}

\newtheorem{examplet}[thm]{Example}

\newtheorem*{acknowledgements}{Acknowledgements}

\begin{document}

\title[Distinguished  representations]{Cohomologically induced distinguished  representations and cohomological test vectors}

\author[B. Sun]{Binyong Sun}
\address{HCMS, HLM, CEMS, Academy of Mathematics and Systems Science, Chinese Academy of Sciences \& School of Mathematical Sciences, University of Chinese Academy of Sciences, Beijing 100190, China
} \email{sun@math.ac.cn}



\begin{abstract}
Let $G$ be a real reductive group, and let $\chi$ be a character of a reductive subgroup $H$ of $G$.  We construct $\chi$-invariant linear functionals on certain cohomologically induced representations of $G$, and show that these linear functionals do not vanish on the bottom layers. Applying this construction, we prove two Archimedean non-vanishing hypotheses, which are vital to the arithmetic study of special values of certain L-functions via modular symbols.
\end{abstract}

 \maketitle


\section{Introduction}

\subsection{Distinguished representations}\label{genrality}
Let $G$ be a real reductive group, namely, it is a Lie group with
the following properties:
\begin{enumerate}
       \item[$\bullet$]
        $\g$ is reductive;
                   \item[$\bullet$]
  $G$ has only finitely many connected components;
         \item[$\bullet$]
            there is a connected closed subgroup of $G$ with
  finite center whose complexified Lie algebra equals  $[\g,\g]$.
\end{enumerate}
Here and henceforth, we use the corresponding lower case Gothic letter to indicate the complexified
Lie algebra of a Lie group. In particular, $\g$ denotes the complexified Lie algebra of $G$.
For applications to the theory of automorphic forms, we are interested in Casselman-Wallach representations of $G$. Recall that a (complex)
representation of a real reductive group is said to be
Casselman-Wallach if it is Fr\'{e}chet, smooth, of moderate growth,
and its underlying Harish-Chandra module is admissible and finitely
generated.  The reader may consult \cite{Cass}, \cite[Chapter 11]{W2} or \cite{BK} for more details about Casselman-Wallach
representations.  To ease notation, we do not
distinguish a representation with its underlying vector space, or a character of a Lie group with its corresponding one-dimensional representation.

Let $H$ be a closed subgroup of $G$, and let  $\chi: H\rightarrow \C^\times$ be a character. By a $\chi$-distinguished representation of $G$, we mean a Casselman-Wallach representation $\pi$ of $G$, together with an $H$-equivariant continuous linear functional $\varphi: \pi\rightarrow \chi$. Distinguished representations are ubiquitous in representation theory and in the theory of automorphic forms.

Fix a Cartan involution $\theta$ of $G$.
From now on we assume that $H$ has only finitely many connected components, and that $\theta(H)=H$.  Then $H$ is also a real reductive group, and $\theta$ restricts to a Cartan involution of $H$. Write
\[
  K:=G^\theta \quad (\textrm{the  fixed point group})\quad \textrm{and}\quad H_{\mathrm c}:=H\cap K,
\]
which are respectively  maximal compact subgroups of $G$ and $H$.

Recall that Casselman-Wallach globalizations establish an equivalence between the category of finitely generated admissible $(\g,K)$-modules and the category of Casselman-Wallach representations of $G$. In particular, every Casselman-Wallach representation $\pi$ of $G$ equals the Casselman-Wallach globalization of its  underlying $(\g,K)$-module $\pi_{[K]}$.
The
restriction induces an injective linear map
\be\label{autinj}
  \Hom_H(\pi,\chi)\hookrightarrow\Hom_{\h,H_{\mathrm c}}(\pi_{[K]},\chi).
\ee

We say that the quadruple $(G, \theta, H, \chi)$ has the automatic continuity property if the map \eqref{autinj} is surjective for all Casselman-Wallach representations $\pi$ of $G$. At least when this is the case, one may study $\chi$-distinguished representations in the purely algebraic setting of $(\g,K)$-modules. The reader is referred to \cite{BK} for  more discussions on the automatic continuity property. It holds at least for symmetric subgroups as in the following theorem.

\begin{thm}\label{autcon0} (\cite[Theorem 1]{BaD} and \cite[Theorem 1]{BrD})
If there is an involutive automorphism $\sigma$ of $G$ which commutes with $\theta$ such that $H$ is an open subgroup of $G^\sigma$, then $(G, \theta, H, \chi)$ has the automatic continuity property.
\end{thm}

\begin{remark}
The proof of Theorem \ref{autcon0} by van den Ban-Delorme
 and Brylinski-Delorme is carried out for trivial $\chi$, but the same proof works in general. The author thanks Patrick Delorme for confirming this.
\end{remark}

\begin{examplet}\label{exm11}
Let $(G, H):=(\GL_{2n}(\R), \Sp_{2n}(\R))$ ($n\geq 1$). For every $g\in G$, let
\be\label{thetaint}
 \theta(g):= g^{-\mathrm t} \qquad \textrm{ (the inverse transpose)}
\ee
and
\[
   \sigma(g):=\left[
  \begin{array}{cc}
   0&1_n \\
   -1_n  &  0 \\
  \end{array}
\right] g^{-\mathrm t} \left[
  \begin{array}{cc}
    0&-1_n \\
   1_n  & 0  \\
      \end{array} \right]\quad (\textrm{$1_n$ is the identity matrix of size $n$}).
\]
Then $\theta$ is a Cartan involution of $G$ stabilizing $H$, and $\sigma$ is an involutive automorphism of $G$ which commutes with $\theta$ such that $G^\sigma=H$.

\end{examplet}

The original motivation of this paper is to settle two
non-vanishing hypotheses which are vital to the arithmetic study of special values of certain L-functions (Theorems \ref{nv1} and \ref{nv2} in this paper).
For this purpose, we will also consider the following two examples of  $(G,H, \theta, \sigma)$.

\begin{examplet}\label{exm12}

Let $(G, H):=(\GL_n(\K), \GL_n(\R))$ ($n\geq 1$), where  $\K$ is a topological field which is topologically isomorphic to $\C$. For every $g\in G$, let
 \[
 \theta(g):=\bar g^{-\mathrm t} \qquad \textrm{ (the inverse of the conjugate transpose)}
\]
and
\[
\sigma(g):=\bar g\qquad\textrm{ (the complex conjugate)}.
\]

\end{examplet}

\begin{examplet}\label{exm13}
 Let $(G, H):=(\GL_{2n}(\R), \GL_n(\R)\times \GL_n(\R))$ ($n\geq 1$).
For every $g\in G$, let  $\theta(g)$ be as in \eqref{thetaint}, and
\[
   \sigma(g):=\left[
  \begin{array}{cc}
   1_n&0 \\
     0& -1_n  \\
  \end{array}
\right] g \left[
  \begin{array}{cc}
   1_n&0 \\
     0& -1_n  \\
      \end{array} \right].
\]
\end{examplet}

Like Example \ref{exm11}, in the above two examples,  $\theta$ is a Cartan involution of $G$ stabilizing $H$, and $\sigma$ is  an involutive automorphism of $G$ which commutes with $\theta$ such that $G^\sigma=H$.
Hence  Theorem \ref{autcon0} applies to all of the above three examples. We remark that in all these three examples, $(G,H)$ is a Gelfand pair, namely (see \cite{AGS, AiS})
\[
  \dim \Hom_H( \pi, \C)\leq 1\quad (\textrm{$\C$ stands for the trivial representation})
\]
for every irreducible Casselman-Wallach representation $\pi$ of $G$.

\subsection{Cohomologically induced distinguished representations}
The main theme of this paper is an algebraic construction of distinguished representations via cohomological induction.

 To be precise, let $\q$ be a  parabolic subalgebra
of $\g$ which is $\theta$-stable, namely, $\theta(\q)=\q$. Here $\theta: \g\rightarrow \g$ denotes the complexified differential of $\theta: G\rightarrow G$. We use ``$\bar\quad$" to indicate the complex conjugation in various contexts. In particular, $\bar \quad:\g\rightarrow \g$ denotes the complex conjugation with respect to the real form $\Lie(G)$ of $\g$. Note that the parabolic subalgebras $\q$ and $\bar \q$ are opposite to each other. Put
\[
 L:= \operatorname N_G(\q)=\operatorname N_G(\bar \q)\qquad (\textrm{the
  normalizers}),
\]
and put $L_{\mathrm c}:=L\cap K$. Then $L$ is a $\theta$-stable real reductive group, and $L_{\mathrm c}$ is a maximal compact subgroup of it.

Denote by $\n$ the nilpotent radical of $\q\cap [\g,\g]$. Then the
parabolic subalgebras $\q$ and $\bar \q$ respectively have Levi decompositions
\be\label{decomq}
  \q=\l \oplus \n\quad\textrm{and}\quad \bar \q=\l \oplus \bar
  \n.
\ee
Put $\n_{\mathrm c}:= \n\cap \k$.

\newcommand{\rc}{\mathrm{c}}
Write $\Pi_{\bar \q, L_\rc}^{\g,K}$ for the $(\dim \n_\mathrm c)$-th left derived functor of the functor
\[
  \oR(\g,K)\otimes_{\oR(\bar \q, L_\rc)} (\,\cdot\,)
\]
from the category of $(\bar \q,L_\rc)$-modules to the category of $(\g,K)$-modules. Here ``$\,\oR\,$" indicates the Hecke algebra of a pair (see \cite[Chapter I, Section 5]{KV}). Let $X$ be an $(\l, L_\rc)$-module, to be viewed as a $(\bar \q, L_\rc)$-module via the trivial action of $\bar \n$.  Then we get the cohomologically induced $(\g,K)$-module $\Pi_{\bar \q, L_\c}^{\g,K}(X)$.

Note that  $L\cap H$ is a $\theta$-stable real reductive group in $L$, and $L_\rc\cap H_\rc=L\cap H\cap K$ is a maximal compact subgroup of it.
The adjoint representation of  the compact group $L_\rc\cap H_\rc$ on the one-dimensional space
\[
  \wedge^{\mathrm{top}} (\h_\rc /(\l_\rc\cap \h_\rc)) \qquad (\textrm{the top degree wedge product})
\]
preserves a real form. Thus it corresponds to a quadratic character of $L_\rc\cap H_\rc$. Consequently, it is also an $(\l\cap \h, L_\rc\cap H_\rc)$-module with the trivial $(\l\cap \h)$-action.
This $(\l\cap \h, L_\rc\cap H_\rc)$-module further corresponds to a quadratic character of $L\cap H$, which we denote by $\varepsilon_{L\cap H}$.
Using this character, we define a character on $L\cap H$ by
\be\label{defchilh}
  \chi_{L\cap H}:=\varepsilon_{L\cap H}\cdot \chi|_{L\cap H}.
\ee

Under the assumptions
 \be\label{q}
  \q+\h=\g \quad\textrm{and}\quad \q\cap \h=\bar \q\cap \h,
\ee
in Section \ref{cohd},  we will construct $\chi$-invariant linear functionals on $\Pi_{\bar \q, L_\rc}^{\g,K}(X)$ from $\chi_{L\cap H}$-invariant linear functionals on $X$, in other words,  we will define a linear map
 \be\label{definvf}
 \Pi_{\bar \q, L_\rc}^{\g,K}:    \Hom_{\l\cap \h,L_\rc\cap H_\rc}(X,\chi_{L\cap H})\rightarrow  \Hom_{\h,H_\rc}(\Pi_{\bar \q, L_\rc}^{\g,K}(X),\chi).
 \ee

 \begin{examplet}\label{qsym00}
 Let $G_1$ be a real reductive group with a Cartan involution $\theta_1$. Suppose $G=G_1\times G_1$ and $\theta=\theta_1\times \theta_1$. Suppose $H$ is the group $G_1$ diagonally embedded in $G$. Let $\q_1$ be a $\theta_1$-stable parabolic subalgebra of $\g_1$. Then $\q:=\q_1\oplus \overline{\q_1}$ satisfies \eqref{q}.
  \end{examplet}

\begin{remark} By using the pair $(G, H)$ of Example \ref{qsym00},  our construction in \eqref{definvf} recovers the construction of Shapovalov forms (\cf \cite[Section VI.4]{KV}) which are used to prove the unitarity of certain cohomologically induced representations.

\end{remark}

 \begin{examplet}\label{qsym}
 Suppose $\sigma$ is as in Theorem \ref{autcon0} so that $H$ is a symmetric subgroup. Let $x$ be an element of the Lie algebra of $K$ such  that $\sigma(x)=-x$ (here $\sigma:\g\rightarrow \g$ denotes the complexified differential of $\sigma: G\rightarrow G$). Suppose $\q=\q_x$ is the sum of the eigenspaces with non-negative eigenvalues of the operator $[x/\sqrt{-1} ,\, \cdot\,]:\g\rightarrow \g$.
 Then $\q$ is a $\theta$-stable parabolic subalgebra satisfying \eqref{q}. In this case,  $\sigma(L)=L$, and $L\cap H$ is a symmetric subgroup of $L$.
  \end{examplet}

\begin{remark}
Let the notation and the assumptions be as in Example \ref{qsym}.
Flensted-Jensen \cite{FJ} and Oshima-Matsuki \cite{OM} classify the discrete series representations of $G$ on $G/H$, namely the irreducible subrepresentations of $\oL^2(G/H)$. In particular, they show that
such a representation exists if and only if the symmetric spaces $G/H$ and $K/H_\rc$ have equal rank.
Assume that  $G$ is connected.  Schlichtkrull
\cite{Sch1} and Vogan \cite{Vog2} prove that all these discrete series representations have  the form $\Pi_{\bar \q, L_\rc}^{\g,K}(X)$, where $\q=\q_x$ for an appropriate $x$ as in Example \ref{qsym}, and $X$ is an appropriate  character of $L$.   The embedding $\Pi_{\bar \q, L_\rc}^{\g,K}(X)\hookrightarrow \oL^2(G/H)$ provided by the works of Flensted-Jensen  and Oshima-Matsuki, followed by the evaluating map at $1\in G/H$, yields an invariant linear functional on $\Pi_{\bar \q, L_\rc}^{\g,K}(X)$, namely an element of
\[
  \Hom_{\h,H_\rc}(\Pi_{\bar \q, L_\rc}^{\g,K}(X),\C)\qquad (\textrm{$\C$ stands for the trivial representation}).
\]
 We expect that  our construction \eqref{definvf} also produces this  invariant linear functional.

 \end{remark}

Besides the symmetric subgroup case of the  above two examples, there are some other interesting cases of $\q$ satisfying \eqref{q}. See \cite[Section 2.2]{Sun} for an example of the Gross-Prasad case.

\begin{remarks} (a) Assume that the $(\l, L_\rc)$-module $X$ is  finitely generated and admissible, and $\varphi\in \Hom_{\l\cap \h,L_\rc\cap H_\rc}(X,\chi_{L\cap H})$ continuously extends to the Casselman-Wallach globalization of $X$. It is natural to ask the following question: does the linear functional $\Pi_{\bar \q, L_\rc}^{\g,K}(\varphi)$ continuously extend  to the Casselman-Wallach globalization of $\Pi_{\bar \q, L_\rc}^{\g,K}(X)$? By Theorem \ref{autcon0}, the answer is yes in the symmetric subgroup case. Due to the lack of a  theory of cohomological inductions in the setting of Casselman-Wallach representations, not much is known beyond this case.

(b) With the notation as in (a), assume that $\Pi_{\bar \q, L_\rc}^{\g,K}(\varphi)$ continuously extends to the Casselman-Wallach globalization $(\Pi_{\bar \q, L_\rc}^{\g,K}(E))^\infty$ of $\Pi_{\bar \q, L_\rc}^{\g,K}(E)$. Then we get a $G$-intertwining linear map
\[
   \begin{array}{rcl}
    (\Pi_{\bar \q, L_\rc}^{\g,K}(E))^\infty&\rightarrow & \Ind_H^G \chi:=\{f\in \con^\infty(G)\mid f(hg)=\chi(h)f(g),\,h\in H, \,g\in G\},\\
      v&\mapsto&(g\mapsto (\Pi_{\bar \q, L_\rc}^{\g,K}(\varphi))(g.v)).
      \end{array}
\]
For unitary $\chi$, it is interesting to know in which cases the image of the above map is contained in the space of square integrable sections.
See \cite{Sch1} for a relevant work in the symmetric subgroup case.

\end{remarks}

If $G$ is compact, then both $H$ and $L$ are compact, and the linear map \eqref{definvf} is also written as
 \be\label{definvf002}
 \Pi_{\bar \q, L}^{\g,G}:    \Hom_{L\cap H}(X,\chi_{L\cap H})\rightarrow  \Hom_{H}(\Pi_{\bar \q, L}^{\g,G}(X),\chi).
 \ee
The following theorem will be proved in Section \ref{secpr}. It asserts that, when $G$ is compact and connected, our construction produces all the $\chi$-invariant linear functionals.

\begin{thm}\label{main} 
Suppose $G$ is compact and connected. Let $X$ be an $L$-module (namely an $(\l, L)$-module). Assume \eqref{q} holds so that the linear map \eqref{definvf002} is defined. Then the linear map \eqref{definvf002}  is surjective.
\end{thm}

In many interesting cases,  the two spaces in \eqref{definvf002}  are both one-dimensional. When this is the case, the surjectivity of the linear map \eqref{definvf002} implies that it is also injective. In general, the following example shows that the linear map \eqref{definvf002} may not be injective.

\begin{examplet}
Suppose $G=\SU(2)$,  $H=\SO(2)$
and  $\chi=\chi_k$ ($k\in \Z$) is the character of $H$ given by
\[
   \left[
  \begin{array}{cc}
   \cos t & \sin t\\
 -\sin t  &  \cos t \\
  \end{array}
\right] \mapsto \cos (kt)+\sqrt{-1} \sin( kt).
\]

Suppose $\q$ is the Borel subalgebra of $\g$ consisting of the upper triangular matrices so that \eqref{q} is satisfied. Then \[
L=\left\{ \left[
  \begin{array}{cc}
   a &0 \\
 0  &  a^{-1} \\
  \end{array}
\right] \mid a\in \oU(1) \right\}=\oU(1).
\]
Suppose $X=X_m$ is a one-dimensional representation of $L$ of weight $m\geq 2$. Then $\Pi_{\bar \q, L}^{\g,G}(X_m)$ is an irreducible representation of $G$ of highest weight $m-2$.

Note that $L\cap H=\{\pm 1\}$ and the character $\varepsilon_{L\cap H}$ is trivial.
If $k$ and $m$ have the same parity, then
\[
  \dim \Hom_{L\cap H}(X_m,\chi_k)=1.
\]
On the other hand,
\[
  \dim \Hom_{H}(\Pi_{\bar \q, L}^{\g,G}(X_m),\chi_k)=0
\]
unless $2-m\leq k\leq m-2$. Therefore, if $k$ and $m$ have the same parity, and $\abs{k}>m-2$, then  \eqref{definvf002} is a linear map
  from a one-dimensional space to a zero space. Thus it is not injective.

\end{examplet}

\subsection{Test vectors in the bottom layers}\label{inttest}

Now we return to the general case when $G$ may or may not be compact.
Given a $\chi$-distinguished representation $(\pi,\varphi)$ of $G$, it is an important (and often hard) problem to find an explicit vector $v_0\in \pi$ such that $\varphi(v_0)\neq 0$. Such a vector is called a test vector of the $\chi$-distinguished representation. For arithmetic applications, we are particularly interested in the case when $\pi$ is an irreducible unitarizable representation with nonzero cohomology, and we hope to find a test vector in $\pi$ which supports the cohomology.

Recall that all irreducible unitary representations with nonzero cohomology are obtained by cohomological induction. At least for these representations, the bottom layers coincide with the minimal $K$-types (in the sense of Vogan), and they have non-trivial contributions to the cohomologies (see \cite{VZ} and \cite{BW}). In this sense, we refer  test vectors in the bottom layers as cohomological test vectors.
In Theorem \ref{nonvmod} of Appendix \ref{appendix}, we will show that the non-vanishing on the bottom layers implies the non-vanishing of certain restriction  maps of cohomology spaces. These restriction maps capture  the Archimedean behaviors of modular symbols, and we call them modular symbols at infinity. The modular symbols are non-trivial and of arithmetic  interest only when the corresponding modular symbols at infinity are nonzero (see  \cite{AG, GR, AS, GHL,KMS} for examples). In Sections \ref{firsta} and \ref{seconda}, we will give two  examples of nonzero modular symbols at infinity, which are respectively used in the arithmetic study of the Asai L-functions \cite{GHL} and the standard L-functions for $\mathrm{GSpin}(2n+1)$ \cite{AG, GR}.

Recall from \cite[Lemma 5.10]{KV} that if $G$ is connected, then so is $L$. To be a little more general, we now assume that $L$ is connected, but allow $G$ to be disconnected. This includes the cases of $(G,L)=(\GL_{2n}(\R), \GL_n(\C))$ or $(\GL_{2n}(\R), (\C^\times)^n)$  which we are considering in this paper. We use a superscript ``$\, ^\circ\, $" to indicate the identity connected component of a Lie group. The connectedness of $L$ implies that $L_\rc^\circ=L_\rc\subset K^\circ$.

Similar to \eqref{decomq}, we have Levi decompositions
\[
  \q_\mathrm c=\l_\rc\oplus \n_\mathrm c\quad\textrm{and}\quad \overline{\q_\mathrm c}=\l_\rc\oplus \overline{
  \n_\mathrm c},
\]
where $\q_\mathrm c:=\q\cap \k$ is a parabolic subalgebra of $\k$, and $\n_\mathrm c:=\n\cap \k$ is  the nilpotent radical of $\q_\mathrm c\cap [\k,\k]$.
Let $E$ be an $L_\rc$-module, namely an $(\l_\rc, L_\rc)$-module. Similar to  the $(\g,K)$-module $\Pi_{\bar \q, L_\rc}^{\g,K}(X)$, we have a  cohomologically induced $K^\circ$-module $\Pi_{\overline{\q_\mathrm c}, L_\rc}^{\k,K^\circ}(E)$.

We say that an  irreducible $L_\rc$-module is dominant (with respect to $\q_{\mathrm c}$) if it is isomorphic to  $\tau^{\n_\mathrm c}$ for an irreducible $K^\circ$-module $\tau$. Here and as usual, a superscript Lie algebra indicates the vectors annihilated by the  Lie algebra action.
        If $E$ is irreducible, then the algebraic version of the Bott-Borel-Weil Theorem (\cf \cite[Corollary 4.160 and Proposition 4.173]{KV}) implies that
        \be\label{bbw}
       \Pi_{\overline{\q_\mathrm c}, L_\rc}^{\k,K^\circ}(E)\cong
          \left\{
          \begin{array}{ll}
            \tau,\quad & \textrm{if $E\otimes \wedge^{\dim \n_\mathrm c}\overline{\n_\mathrm c}  \cong \tau^{\n_\mathrm c}$ is dominant};\\
            \{0\} , \quad &\textrm{if $E\otimes \wedge^{\dim \n_\mathrm c} \overline{\n_{\mathrm c}}$ is not dominant.}
                      \end{array}
          \right.
        \ee

Note that \eqref{q} implies that
\be\label{qchc}
 \q_{\mathrm c}+\h_\rc=\k \quad\textrm{and}\quad \q_{\mathrm c}\cap \h_\rc= \overline{\q_{\mathrm c}}\cap \h_\rc.
 \ee
Similar to  \eqref{definvf}, the equalities in \eqref{qchc} enable us to define a linear map
  \be\label{definvf2}
   \Pi_{\overline{\q_\mathrm c}, L_\rc}^{\k,K^\circ}:    \Hom_{ L_\rc \cap H_\rc}(E,\chi_{L\cap H})\rightarrow  \Hom_{K^\circ\cap H_\rc}(\Pi_{\overline{\q_\mathrm c},L_\rc}^{\k,K^\circ}(E),\chi).
 \ee
By Theorem \ref{main}, this map is surjective.

Recall that $X$ is an $(\l, L_\rc)$-module.
 Fix a homomorphism $\phi\in \Hom_{L_\rc}(E, X)$. Then  we have the bottom layer map (see Section \ref{secb})
\[
   \beta(\phi)\in \Hom_{K^\circ}(\Pi_{\overline{\q_\mathrm c}, L_\rc}^{\k,K^\circ}(E),\Pi_{\bar \q, L_\rc}^{\g,K}(X)).
\]

The following proposition asserts that our construction of invariant functionals is compatible with the bottom layer map.

\begin{prpt}\label{commb0}(see Proposition \ref{commb})
Suppose \eqref{q} is satisfied and $L$ is connected.  Let $X$ be an $(\l, L_\rc)$-module and let   $E$ be an $L_\rc$-module. Then the diagram
 \begin{equation}\label{cdk000}
 \begin{CD}
           \Pi_{\overline{\q_\mathrm c}, L_\rc}^{\k,K^\circ}(E) @>\beta(\phi) >> \Pi_{\bar \q, L_\rc}^{\g,K}(X) \\
            @VV  \Pi_{\overline{\q_\mathrm c}, L_\rc}^{\k,K^\circ}(\varphi\circ \phi) V           @VV\Pi_{\bar \q, L_\rc}^{\g,K}(\varphi) V\\
           \chi @= \chi\\
  \end{CD}
\end{equation}
commutes, for all $\phi\in \Hom_{L_\rc}(E, X)$ and $\varphi\in \Hom_{\l\cap \h, L_\rc \cap H_\rc}(X,\chi_{L\cap H})$.
\end{prpt}


As mentioned earlier, we hope to find test vectors in the bottom layers. More precisely, at least in some interesting cases, we want  to show that the $\chi$-invariant linear functional $\Pi_{\bar \q, L_\rc}^{\g,K}(\varphi)$ does not vanish on the bottom layers.
Proposition \ref{commb0} reduces this problem to the case of compact connected groups. Consequently, we may apply  Theorem \ref{main} to get the following result.

\begin{thm}\label{mainf}
Suppose \eqref{q} is satisfied and $L$ is connected. Let $X$ be an $(\l, L_\rc)$-module, and let   $E$ be an irreducible $L_\rc$-submodule of $X$ such that
\[
\dim \Hom_{ L_\rc\cap H_\rc }(E,\chi_{L\cap H})=\dim \Hom_{K^\circ\cap H_\rc}(\Pi_{\overline{\q_\mathrm c}, L_\rc}^{\k,K^\circ}(E),\chi).
\]
If there is a $\chi_{L\cap H}$-invariant linear functional on $X$ which does not vanish on $E$, then there is a $\chi$-invariant linear functional on $\Pi_{\overline{\q}, L_\rc}^{\g,K}(X)$ which does not vanish on $\Pi_{\overline{\q_\mathrm c}, L_\rc}^{\k,K^\circ}(E)$.

\end{thm}

Note that the assumptions of Theorem \ref{mainf} imply that the $K^\circ$-module  $\Pi_{\overline{\q_\mathrm c}, L_\rc}^{\k,K^\circ}(E)$ is nonzero, and hence irreducible. It is  viewed as a $K^\circ$-submodule of $\Pi_{\overline{\q}, L_\rc}^{\g,K}(X)$ via the bottom layer map which is injective (see \cite[Chapter V, Section 6]{KV}).

\begin{examplet}\label{exa110}
Let $(G, H, \theta, \sigma)$ be as in Example \ref{exm11} so that $G=\GL_{2n}(\R)$ and $H=\Sp_{2n}(\R)$.
 Combining \cite[Theorem A]{GOSS} and \cite[Theorem 1.1]{AOS}, the set of $H$-distinguished unitarizable irreducible Casselman-Wallach representations of $G$ is determined (see also \cite{GSS}). Here an irreducible Casselman-Wallach representation of $G$ is said to be $H$-distinguished if there is a  nonzero $H$-invariant continuous  linear functional on it.  A major step in the determination is to show that if $n$ is even, then every Speh representation of $G$ is $H$-distinguished. This is done in \cite[Proposition 4.0.2]{GOSS} by global method, and in  \cite[Theorem A]{GSS} by analytic method. This also follows from Theorem \ref{mainf} as explained in what follows.

For $k\geq 1$, denote by $D_k$ the unique (up to isomorphism) irreducible Cassleman-Wallach representation of $\GL_2(\R)$ in the relative discrete series which has the same infinitesimal character as that of $\oS^{k-1}(\C^2)\otimes \abs{\det}^{-\frac{k-1}{2}}$. Here $\C^2$ stands for the standard representation of $\GL_2(\R)$, and ``$\oS^{k-1}$" indicates the $(k-1)$-th symmetric power. The Speh representation $J_{2n, k}$ of $\GL_{2n}(\R)$ is  the unique irreducible quotient of
\[
  \Ind_{P_{2,2,\cdots, 2}}^G ((D_k\otimes \abs{\det}^{\frac{n-1}{2}})\widehat \otimes (D_k\otimes \abs{\det}^{\frac{n-3}{2}})\widehat \otimes \cdots \widehat \otimes (D_k\otimes \abs{\det}^{\frac{1-n}{2}})).
\]
Here and henceforth, $P_{2,2,\cdots, 2}$ denotes the upper triangular parabolic subgroup of type $(2,2,\cdots, 2)$, ``$\widehat{\otimes}$" stands for the completed projective tensor product, and ``$\Ind$" stands for the normalized smooth induction. Let $y$ be a skew-symmetric matrix in $\gl_{2n}(\R)$ such that $y^2=-1_{2n}$. Define $\q_y$  as in Example \ref{qsym} and suppose $\q=\q_y$. Define a  unitary character
\be\label{unicn}
\kappa_n: \GL_n(\C)\rightarrow \C^\times, \quad z\mapsto \frac{\det( z)}{\abs{\det(z)}}.
\ee
View $y$ as a complex structure so that $\R^{2n}$ is viewed as an $n$-dimensional  complex vector space. Then  $L\cong \GL_n(\C)$ and we view $\kappa_n$ as a character on $L$ via this isomorphism.  It is known that the underlying $(\g,K)$-module of $J_{2n, k}$ is isomorphic to $\Pi_{\bar \q, L_\rc}^{\g,K}(\kappa_n^{k+n})$ (see \cite[Section 4]{Sp}).

Now suppose $n=2m$ is even and
\[
  y=\left[
  \begin{array}{cccc}
   0&1_m&0&0 \\
   -1_m  &  0 & 0&0 \\
    0&0&0&-1_m \\
   0  &  0 &1_m&0 \\
  \end{array}
\right].
\]
Then  $\sigma(y)=-y$ and $\q$ satisfies \eqref{q}. Note that
\[
K^\circ=\SO(2n), \quad H_\rc\cong \oU(n), \quad L_\rc\cong \oU(n),\quad L\cap H\cong \Sp_{2m}(\C),  \quad L_\rc \cap H_\rc \cong \Sp(m).
\]
 By \eqref{bbw}, $\tau:=\Pi_{\overline{\q_\mathrm c}, L_\rc}^{\k,K^\circ}(\kappa_n^{k+n})$ is an irreducible representation of $\SO(2n)$ of highest weight $(k+1, k+1, \cdots, k+1)$.
Cartan-Helgason Theorem (\cf \cite[Chapter V, Theorem 4.1]{Hel}) implies that
\[
  \dim \Hom_{H_\rc}(\tau,\C)=1.
\]
Then it is clear from Theorem \ref{mainf} and Theorem \ref{autcon0} that there is an $H$-invariant continuous linear functional on $J_{2n,k}$ which does not vanish on $\tau$.
We remark that neither the  global method of \cite{GOSS} nor the  analytic method of \cite{GSS} proves the non-vanishing of the invariant linear functional on the minimal $K$-type.
\end{examplet}

In Section \ref{cohd}, we will explain the general construction of
cohomologically induced distinguished representations.
Section \ref{secpr} is devoted to a proof of Theorem \ref{main}. In Section \ref{firsta} and Section \ref{seconda}, we will give two arithmetically interesting applications of our construction.  In Appendix \ref{appendix}, a general non-vanishing result is proved for modular symbols at infinity, which contains Theorem \ref{nv1} and Theorem \ref{nv2} as special cases.

\begin{acknowledgements}
The author thanks Michael Harris for the suggestion to work on the non-vanishing hypothesis for Asai L-functions (Theorem \ref{nv1}),  and thanks Dihua Jiang for the suggestion to work on the non-vanishing hypothesis for L-functions for $\mathrm{GSpin}(2n+1)$ (Theorem \ref{nv2}).  He thanks Patrick Delorme for confirming the automatic continuity theorem in the twisted case (Theorem
\ref{autcon0}), and thanks Wee Teck Gan for the suggestion to write down a general criterion (as in Theorem \ref{nonvmod}) for the non-vanishing of modular symbols at infinity. He also thanks Hongyu He, Fabian Januszewski and Chen-Bo Zhu for helpful
discussions. Part of the work was done when the author participated
in the ``Analysis on Lie groups" program at Max Planck Institute for
Mathematics, in 2011. The author thanks the organizers for the invitation and
thanks Max Planck Institute for Mathematics for their hospitality. He also thanks the referees for helpful suggestions which resulted in the improvement of the paper, in particular, for the suggestion to consider the pair $(\GL_{2n}(\R), \Sp_{2n}(\R))$ of Example \ref{exa110}.
The work was supported in part by the National Natural Science Foundation of China (No.  11525105, 11688101, 11621061 and 11531008).
\end{acknowledgements}



\section{Cohomologically induced distinguished representations}\label{cohd}

\subsection{Bernstein functors and Zuckerman functors}\label{secbz}
We begin with recalling some basic facts concerning Bernstein functors and Zuckerman functors.
Let $(\g,K)$ be a pair as in \cite[Section I.4]{KV}, namely,
\begin{itemize}
  \item $\g$ is a finite-dimensional Lie algebra over $\C$;
  \item $K$ is a compact Lie group whose complexified Lie algebra $\k$ is identified with a Lie subalgebra of $\g$;
  \item there is given an action $\Ad$ of $K$ on $\g$ by automorphisms which extends the adjoint representation of $K$;
  \item the differential of the action $\Ad$ equals the adjoint action of $\k$ on $\g$.
\end{itemize}
Let $L_\rc$ be a closed subgroup of $K$. Denote by $\CB$ the Bernstein functor
\[
   \oR(\g,K)\otimes_{\oR(\g, L_\rc)} (\,\cdot\,)
\]
from the category of $(\g,L_\rc)$-modules to the category of $(\g,K)$-modules. Write $\CB_j$ for its $j$-th left derived functor ($j\in \Z$).
Likewise, denote by $\CZ$ the Zuckerman functor
\[
   \Hom_{\oR(\g, L_\rc)}(\oR(\g,K), (\,\cdot\,))_{\textrm{$K$-finite}}
\]
from the category of $(\g,L_\rc)$-modules to the category of $(\g,K)$-modules, where ``$\textrm{$K$-finite}$" indicates the space of $K$-finite vectors. Write $\CZ^j$ for its $j$-th right derived functor ($j\in \Z$).

As in \cite{DV}, in order to describe the functor $\CB_j$ more explicitly, for each $(\g,L_\rc)$-module $V_0$, we introduce a linear action
\begin{equation}\label{act01}
  (\k,L_\rc)\times (\g,K) \curvearrowright  \oR(K)\otimes V_0\qquad (\oR(K):=\oR(\k,K)),
\end{equation}
as follows:
\begin{itemize}
  \item the pair $(\k,L_\rc)$ acts by the tensor
product of the right translation on  $\oR(K)$ and the restriction of
the $(\g,L_\rc)$-action on $V_0$;
  \item the group $K$ acts on
$\oR(K)\otimes V_0$ through the left translation on $\oR(K)$;
  \item the Lie algebra $\g$ acts on $\oR(K)\otimes V_0$ such that
  \begin{equation}\label{actg0}
  \int_K f(k) \, \od\,\!(X.(\mu\otimes v))(k)=\int_K f(k)(\Ad_{k^{-1}}X).v\, \od\!\mu(k),
\end{equation}
for all $X\in \g, \,\mu\in \oR(K),\, v\in V_0$, and $f\in \C[K]$.
\end{itemize}
Here  
$\C[K]$ denotes the space of all left $K$-finite (or
equivalently, right $K$-finite) smooth functions on $K$. (Similar
notation will be used for other compact Lie groups.) In the left hand side of \eqref{actg0}, we view $\oR(K)\otimes V_0$ as a space of $V_0$-valued measures on $K$.

Under these actions, $\oR(K)\otimes V_0$ becomes a
$(\k,L_\rc)$-module as well as a weak $(\g,K)$-module (see
\cite[Chapter I, Section 5]{KV} for the notion of weak
$(\g,K)$-modules). Furthermore, the $(\k,L_\rc)$-action and the
$(\g,K)$-action commute with each other, and
\begin{equation}\label{cohom0}
 \CB_j(V_0)=\oH_j(\k, L_\rc;\oR(K)\otimes V_0), \quad
  j\in \Z,
\end{equation}
as $(\g,K)$-modules (\cf \cite[Section III.3]{KV}). In particular the homology space $\oH_j(\k, L_\rc;\oR(K)\otimes V_0)$ is not only a weak $(\g,K)$-module, but also a $(\g,K)$-module. The reader is referred to \cite[(2.126)]{KV} and \cite[(2.127)]{KV} for the explicit complexes which  respectively compute the relative Lie algebra homology spaces and  the relative Lie algebra cohomology spaces.

Similarly, in order to describe the functor $\CZ^j$ more explicitly,  for each $(\g,L_\rc)$-module $V$,
 we introduce a linear  action
\begin{equation}\label{act1}
  (\k,L_\rc)\times (\g,K) \curvearrowright \C[K]\otimes V,
\end{equation}
as follows:
\begin{itemize}
  \item the pair $(\k,L_\rc)$ acts by the tensor
product of the left translation on  $\C[K]$ and the restriction of
the $(\g,L_\rc)$-action on $V$;
  \item the group $K$ acts on
$\C[K]\otimes V$ through the right translation on $\C[K]$;
  \item the Lie algebra $\g$ acts by
\begin{equation}\label{actg}
  (X.f)(k):=(\Ad_{k}X).f(k), \quad k\in K, \,f\in \C[K]\otimes V.
\end{equation}
\end{itemize}
Here and as usual, $\C[K]\otimes V$ is identified
with a space of $V$-valued functions on $K$.
Then similar to \eqref{cohom0}, we have that
\be\label{cohom00}
   \CZ^j(V)=\oH^j(\k, L_\rc;\C[K]\otimes V), \quad
  j\in \Z,
\ee
as $(\g,K)$-modules.

Recall the following Zuckerman Duality Theorem.
\begin{thm}\label{zd}(see \cite[Corollary 3.7]{KV})
For every $(\g,L_\rc)$-module $V_0$ and every $j\in \Z$, there is a canonical isomorphism
\be\label{zuckd}
   \CB_j(V_0)\cong \CZ^{m-j}(\wedge^{m} \k/\l_\rc\otimes V_0)\qquad(m:=\dim \k/\l_\rc)
\ee
of $(\g,K)$-modules. Here $\wedge^{m}\k/\l_\rc$ is viewed as a $(\g,L_\rc)$-module such that $\g$ acts trivially, and $L_\rc$ acts through the adjoint representation.

\end{thm}
\begin{proof}
One checks that the linear isomorphism
\[
   \begin{array}{rcl}
     I_j: \wedge^j(\k/\l_\rc)\otimes \oR(K)\otimes V_0 &\rightarrow & \Hom_\C(\wedge^{m-j}(\k/\l_\rc), \C[K]\otimes \wedge^{m} \k/\l_\rc\otimes V_0),\\
         \omega\otimes f\mu_K\otimes v&\mapsto & (\omega'\mapsto  f^\vee\otimes (\omega\wedge \omega') \otimes v )
   \end{array}
\]
is $L_\rc$-equivariant and $(\g,K)$-equivariant, and $\{(-1)^\frac{j(j+1)}{2}I_j\}_{j\in \Z}$ restricts to a morphism of the chain complexes which compute \eqref{cohom0} and \eqref{cohom00}. Here $\mu_K$ denotes the Haar measure on $K$ such that $K^\circ$  has volume $1$, and $f^\vee$ is the function on $K$ such that  $f^\vee(x)=f(x^{-1})$ for all $x\in K$.
Therefore $I_j$ induces an isomorphism \eqref{zuckd}. See \cite[Chapter III]{KV} for more details.
\end{proof}

\subsection{Some one-dimensional modules}
Now we retain the notation of the Introduction. Recall that we have the Lie groups
\[
\xymatrix{
                & L \ar@{-}[dr]             \\
  L\cap H \ar@{-}[ur] \ar@{-}[dr] & & L_\rc \ar@{-}[dl]         \\
    & L_\rc\cap H_\rc&
     }
\xymatrix{
                &             \\
  &  \subset    &     \\
    &
     }
 \xymatrix{
                & G \ar@{-}[dr]             \\
  H \ar@{-}[ur] \ar@{-}[dr] & & K \,,\ar@{-}[dl]         \\
    & H_\rc &
     }
\]
and the Lie algebras $\q=\l\oplus \n$ and $\q_\mathrm c=\l_\rc\oplus \n_\mathrm c$. We first allow $L$ to be disconnected.

Recall from \eqref{q} that
\be\label{q2}
  \q+\h=\g \quad\textrm{and}\quad \q\cap \h=\bar \q\cap \h.
\ee

\begin{lemt}\label{tranver}
One has that
\begin{equation}\label{ghq1}
  \g=\q+\h=\bar \q+\h, \quad \q\cap \h=\bar \q \cap \h=\l\cap \h,
\end{equation}
and
\begin{equation}\label{ghq2}
 \k=\q_{\mathrm c}+\h_\rc=\overline{\q_{\mathrm c}}+\h_\rc,\quad \q_{\mathrm c}\cap \h_\rc=\overline{\q_{\mathrm c}}\cap \h_\rc=\l_\rc\cap \h_\rc.
\end{equation}
\end{lemt}
\begin{proof}
The equalities of \eqref{ghq1} is an obvious consequence of the assumption \eqref{q2}. The equalities of \eqref{ghq2} is implied by \eqref{ghq1} since $\h$, $\q$, $\bar \q$ and $\l$ are all $\theta$-stable.
\end{proof}

\begin{lemt}\label{defs}
As representations of $L_\rc\cap H_\rc$,
\be\label{hrc0}
   \h_\rc/(\l_\rc\cap \h_\rc)\cong\n_\mathrm c\cong \overline{\n_\mathrm c}\cong \k/(\l_\rc+\h_\rc),
\ee
and they are all self-dual.
\end{lemt}
\begin{proof}
By \eqref{ghq2}, we have that
\begin{equation}\label{isocc}
  \h_\rc/(\l_\rc\cap \h_\rc)=\h_\rc/( \overline{ \q_{\mathrm c}}\cap \h_\rc)\cong \k/\overline{\q_{\mathrm c}}\cong \n_{\mathrm c}\cong \q_\mathrm c/\l_\rc\cong \k/(\l_\rc+\h_\rc).
\end{equation}
Similarly,
\begin{equation}\label{isocc}
  \h_\rc/(\l_\rc\cap \h_\rc) \cong \overline{ \n_{\mathrm c}}\cong  \k/(\l_\rc+\h_\rc).
\end{equation}
Note that $\n_\mathrm c$ and $\overline{\n_\mathrm c}$ are dual to each other under the Killing form of $\k$. Therefore the lemma follows.
\end{proof}

Write $S:=\dim \n_\mathrm c$ for simplicity. One important consequence of our main assumption \eqref{q2} is the following ``degree match":
\[
   \dim (\h_\rc/(\l_\rc \cap \h_\rc))=S\qquad (\textrm{this follows from  \eqref{hrc0}}).
\]

Recall from the Introduction that $\chi$ is a character of $H$.   View $\wedge^S \h_\rc/(\l_\rc \cap \h_\rc)$ as an $(\h, L_\rc\cap H_\rc)$-module such that $\h$ acts trivially and $L_\rc\cap H_\rc$ acts through the adjoint representation.
Define a one-dimensional  $(\h,L_\rc\cap H_\rc)$-module
\[
  \xi:=(\wedge^S \h_\rc/(\l_\rc \cap \h_\rc) )\otimes \chi|_{(\h,L_\rc\cap H_\rc)}.
\]

Note that the adjoint action of $L_\rc$ on $\wedge^{2S}\k/\l_\rc$ is trivial. We view $\wedge^{2S}\k/\l_\rc$ as a trivial  $(\g,L_\rc)$-module, and likewise view $\wedge^{2S}(\k/\l_\rc)^\vee$ as a trivial  $(\g,L_\rc)$-module. Here and henceforth, a superscript ``$\,\,^\vee\,$" indicates the dual module.
Put
\[
  \xi_0:= (\wedge^{2S}(\k/\l_\rc)^\vee)\otimes \xi.
\]
Then the one-dimensional module $(\xi_0)|_{(\l\cap \h, L_\rc\cap H_\rc)}$ corresponds to  the character $\chi_{L\cap H}$ of $L\cap H$ (which is defined in the Introduction).
We identify $\xi_0$ with $\chi_{L\cap H}$ in what follows.

\subsection{The  construction}

Let $X$ be an $(\l,L_\rc)$-module, and let  $$\varphi\in \Hom_{\l\cap \h,L_\rc\cap H_\rc}(X,\chi_{L\cap H}).$$ View $X$ as a $(\bar \q, L_\rc)$-module such that $\bar \n$ acts trivially on it. Put
\[
  V_0:=\oU(\g)\otimes_{\oU(\bar \q)}X\qquad(\textrm{``$\oU$" indicates the universal enveloping algebra}).
  \]
This is a $(\g,L_\rc)$-module with $\g$ acts by the left multiplication, and
$L_\rc$ acts by the tensor product of its adjoint action on $\oU(\g)$
and its given action on $X$.

\begin{lemt}\label{phic}
There is a unique $(\h, L_\rc\cap H_\rc)$-equivariant linear map
\be\label{psi0}
\psi_0:
V_0\rightarrow \xi_0
\ee
 which extends $\varphi$.
\end{lemt}
\begin{proof}
By \eqref{ghq1}, we have that
\[
   V_0=\oU(\g)\otimes_{\oU(\bar \q)}X=\oU(\h)\otimes_{\oU(\l\cap \h)}X
  \]
as an $(\h,L_\rc\cap H_\rc)$-module. Therefore the lemma is a form of Frobenious
reciprocity.
\end{proof}

Define a $(\g,L_\rc)$-module
\[
  V:= \wedge^{2S} \k/\l_\rc \otimes V_0.
\]
The linear functional $\psi_0$ of \eqref{psi0} induces an $(\h,L_\rc\cap H_\rc)$-equivariant linear functional
\be\label{psi}
\psi:=1_{\wedge^{2S} \k/\l_\rc}\otimes \psi_0:
V\rightarrow \xi.
\ee

Similar to the action
\begin{equation}\label{act11}
  (\k,L_\rc)\times (\g,K) \curvearrowright \C[K]\otimes V
\end{equation}
of \eqref{act1}, based on the $(\h,L_\rc\cap H_\rc)$-action on $\xi$, we define
an action
\be\label{act2}
  (\h_\rc,L_\rc\cap H_\rc)\times (\h,H_\rc) \curvearrowright \C[H_\rc]\otimes \xi.
\ee
Note that there is a component-wise inclusion
\[
   (\h_\rc,L_\rc\cap H_\rc)\times (\h,H_\rc)\subset (\k,L_\rc)\times (\g,K),
\]
and the map
\begin{equation}\label{map1}
  r_{K,H_\rc}\otimes \psi : \C[K]\otimes V \rightarrow\C[H_\rc]\otimes \xi
\end{equation}
is $(\h_\rc,L_\rc\cap H_\rc)\times (\h,H_\rc)$-equivariant, where $r_{K,H_\rc}$ denotes  the
restriction map.

Write $\tilde \xi$ for the module
\be \label{tnu}
  (\h_\rc,L_\rc\cap H_\rc)\times (\h,H_\rc) \curvearrowright (\wedge^S\h_\rc/(\l_\rc \cap \h_\rc))|_{(\h_\rc,L_\rc\cap H_\rc)}\otimes \chi|_{(\h,H_\rc)}.
\ee
It equals  $\xi$ as a vector space.
\begin{lemt}\label{intg}
The linear map
\begin{equation}\label{map2}
 \begin{array}{rcl}
   \C[H_\rc]\otimes \xi &\rightarrow &\tilde \xi,\medskip\\
      f&\mapsto& \int_{H_\rc} \chi(c)^{-1}\,f(c)\,\od\! c
 \end{array}
\end{equation}
is $(\h_\rc,L_\rc\cap H_\rc)\times (\h,H_\rc)$-equivariant, where $\od\! c$ is the Haar measure on $H_\rc$ such that  $K^\circ \cap H_\rc$ has volume $1$, and as usual, $\C[H_\rc]\otimes \xi$ is
viewed as a space of $\xi$-valued functions on $H_\rc$.
\end{lemt}
\begin{proof}
This is routine to check.
\end{proof}

\begin{lemt}
One has an identification
\begin{equation}\label{cohomo4}
  \oH^S(\h_\rc,L_\rc\cap H_\rc; \tilde \xi)=\chi
\end{equation}
of $(\h,H_\rc)$-modules.
\end{lemt}
\begin{proof}
Poincar\'e duality (\cf \cite[Corollary 3.6]{KV}) implies that
\[
  \oH^S(\h_\rc,L_\rc\cap H_\rc;  \wedge^S\h_\rc/(\l_\rc \cap \h_\rc))=\C.
\]
Thus
\[
  \oH^S(\h_\rc,L_\rc\cap H_\rc; \tilde \xi)=\oH^S(\h_\rc,L_\rc\cap H_\rc; \wedge^S\h_\rc/(\l_\rc \cap \h_\rc))\otimes \chi=\C\otimes \chi=\chi.
\]
\end{proof}

Finally, we define the $(\h,H_\rc)$-equivariant linear functional
\be\label{pis}
  \Pi_{\bar \q, L_\rc}^{\g,K}(\varphi): \Pi_{\bar \q, L_\rc}^{\g,K}(X)\rightarrow \chi
\ee
to be the composition of the following maps:
\begin{eqnarray}
\label{defpiphi}  \Pi_{\bar \q, L_\rc}^{\g,K}(X) &=& \oH_S(\k,L_\rc; \,\oR(K)\otimes V_0) \\
 \nonumber &\rightarrow &\oH^S(\k,L_\rc; \,\C[K]\otimes V) \\
 \nonumber &\rightarrow &\oH^S(\h_\rc,L_\rc\cap H_\rc; \,\C[H_\rc]\otimes   \xi)\\
\nonumber  &\rightarrow & \oH^S(\h_\rc,L_\rc\cap H_\rc; \,\tilde \xi)=\chi.
\end{eqnarray}
Here the first arrow is the Zuckerman duality isomorphism as in Theorem \ref{zd}, the second arrow is the restriction of cohomology induced by the map \eqref{map1}, and the third arrow is the linear map induced by the map \eqref{map2}.

It is clear that \eqref{pis} yields a linear map
\be\label{definepip}
 \Pi_{\bar \q, L_\rc}^{\g,K}:    \Hom_{\l\cap \h,L_\rc\cap H_\rc}(X,\chi_{L\cap H})\rightarrow  \Hom_{\h,H_\rc}(\Pi_{\bar \q, L_\rc}^{\g,K}(X),\chi).
\ee

\begin{remark}
The construction of $\Pi_{\bar \q, L_\rc}^{\g,K}(\varphi)$ is functorial in the following sense: for all commutative diagram
 \[
 \begin{CD}
            X_1 @>\eta>> X_2 \\
            @VV \varphi_1 V           @VV\varphi_2 V\\
           \chi_{L\cap H} @= \chi_{L\cap H},\\
  \end{CD}
\]
 where $X_i$ is an $(\l,L_\rc)$-module,  $\varphi_i\in \Hom_{\l\cap \h,L_\rc\cap H_\rc}(X_i, \chi_{L\cap H})$ ($i=1,2$), and $\eta\in \Hom_{\l,L_\rc}(X_1,X_2)$,  the diagram
 \[
 \begin{CD}
           \Pi_{\bar \q, L_\rc}^{\g,K}(X_1) @>\Pi_{\bar \q, L_\rc}^{\g,K}(\eta)>> \Pi_{\bar \q, L_\rc}^{\g,K}(X_2) \\
            @VV \Pi_{\bar \q, L_\rc}^{\g,K}(\varphi_1) V           @VV\Pi_{\bar \q, L_\rc}^{\g,K}(\varphi_2) V\\
           \chi @= \chi\\
  \end{CD}
\]
commutes.
\end{remark}

\subsection{Bottom layers}\label{secb}
Now assume that $L$ is connected so that $L_\rc=L_\rc^\circ\subset K^\circ$.
Recall from the Introduction that  $E$ is an $L_\rc$-module and  $\phi\in \Hom_{L_\rc}(E, X) $.
The homomorphism $\phi$ induces a $(\k,L_\rc)\times (\k,K^\circ)$-equivariant linear map
\[
\begin{array}{rcl}
  \oR(K^\circ)\otimes (\oU(\k)\otimes_{\oU(\overline{\q_\mathrm c})} E)&\rightarrow &\oR(K^\circ)\otimes (\oU(\g)\otimes_{\oU(\bar \q)} X),\\
    \mu\otimes X\otimes v&\mapsto & \mu\otimes X\otimes \phi(v).
\end{array}
\]
Taking the relative Lie algebra cohomologies, we get the bottom layer map
\begin{eqnarray}
 \nonumber &  &\Pi_{\overline{ \q_\rc}, L_\rc}^{\k,K^\circ}(E)=\oH_S(\k,L_\rc; \oR(K^\circ)\otimes (\oU(\k)\otimes_{\oU(\overline{\q_\mathrm c})} E)) \\
  \label{beta0} & \rightarrow  &\Pi_{\bar \q, L_\rc}^{\g,K^\circ}(X)=\oH_S(\k,L_\rc; \oR(K^\circ)\otimes (\oU(\g)\otimes_{\oU(\bar \q)} X)).
\end{eqnarray}

The extension by zero map
\[
   \oR(K^\circ )\rightarrow \oR(K)
\]
induces an injective linear map
\begin{eqnarray}
 \nonumber &&\Pi_{\bar \q, L_\rc}^{\g,K^\circ }(X)=\oH_S(\k,L_\rc; \oR(K^\circ )\otimes (\oU(\g)\otimes_{\oU(\bar \q)} X))\\
\label{inj010}  &\hookrightarrow & \Pi_{\bar \q, L_\rc}^{\g,K}(X)=\oH_S(\k,L_\rc; \oR(K)\otimes (\oU(\g)\otimes_{\oU(\bar \q)} X)).
 \end{eqnarray}
Now the composition of \eqref{beta0} and \eqref{inj010} yields the following bottom layer map:
\[
   \beta(\phi):  \Pi_{\overline{ \q_\rc}, L_\rc}^{\k,K^\circ}(E)\rightarrow  \Pi_{\bar \q, L_\rc}^{\g,K}(X).
\]

Similar to \eqref{definepip} and using \eqref{ghq2}, we have a linear map
\be\label{definepipc}
\Pi_{\overline{ \q_\rc}, L_\rc}^{\k,K^\circ}:    \Hom_{L_\rc \cap H_\rc}(E,\chi_{L\cap H})\rightarrow  \Hom_{K^\circ\cap H}( \Pi_{\overline{ \q_\rc}, L_\rc}^{\k,K^\circ}(E),\chi).
\ee

\begin{prpt}\label{commb}
Assume \eqref{q2} holds and $L$ is connected. Then the diagram
 \begin{equation}\label{cdk00}
 \begin{CD}
            \Pi_{\overline{ \q_\rc}, L_\rc}^{\k,K^\circ}(E) @>\beta(\phi) >> \Pi_{\bar \q, L_\rc}^{\g,K}(X) \\
            @VV \Pi_{\overline{\q_\rc}, L_\rc}^{\k,K^\circ}(\varphi\circ \phi) V           @VV\Pi_{\bar \q, L_\rc}^{\g,K}(\varphi) V\\
           \chi @= \chi\\
  \end{CD}
\end{equation}
commutes for all $\phi\in \Hom_{L_\rc}(E,X)$, and $\varphi\in \Hom_{\l\cap \h,L_\rc\cap H_\rc}(X,\chi_{L\cap H})$.

\end{prpt}
\begin{proof}
This is routine to check. We omit the details.
\end{proof}

\section{A proof of Theorem \ref{main}}\label{secpr}

In this section, suppose $G$ is connected and compact. Then $K=G$, $H_\rc=H$, $\q_\rc=\q$, $\n_\rc=\n$ and $L=L_\rc=L_\rc^\circ$.

Recall that we assume
\be\label{qhg22}
  \q+\h=\g \quad\textrm{and}\quad \q\cap \h=\bar \q\cap \h.
\ee
Our goal is to prove the following Theorem (Theorem \ref{main} of the Introduction).
\begin{thm}\label{mainp}

If $G$ is connected and compact, then the linear map
\be\label{sur1}
 \Pi_{\bar \q, L}^{\g,G}:    \Hom_{L\cap H}(X,\chi_{L\cap H})\rightarrow  \Hom_{H}(\Pi_{\bar \q, L}^{\g,G}(X),\chi)
\ee
is surjective for every $L$-module $X$.
\end{thm}

\subsection{The $G$-module $\Pi_{\bar \q, L}^{\g,G}(X)$}

Recall the $(\g,L)$-module
\[
 V=\wedge^{2S} \g/\l\otimes (\oU(\g)\otimes_{\oU(\bar \q)}X),
  \]
and write
\[
   \widetilde \Pi_{\bar \q, L}^{\g,G}(X):= \oH^S(\g,L; \,\C[G]\otimes V).
\]
It is isomorphic to $\ \Pi_{\bar \q, L}^{\g,G}(X)$ by  the Zuckerman duality isomorphism. The last two arrows of \eqref{defpiphi}  yield a linear map
\be\label{sur2}
 \widetilde \Pi_{\bar \q, L}^{\g,G}:    \Hom_{L\cap H}(X,\chi_{L\cap H})\rightarrow  \Hom_{H}(\widetilde \Pi_{\bar \q, L}^{\g,G}(X),\chi).
\ee
It is clear that the map \eqref{sur1} is surjective if and only if so is the map \eqref{sur2}.

Without loss of generality, assume that the $L$-module $X$ is irreducible. If $\widetilde \Pi_{\bar \q, L}^{\g,G}(X)$ is zero, then Theorem \ref{mainp} is trivial. So we assume it is nonzero.
Then by \eqref{bbw}, $\widetilde \Pi_{\bar \q, L}^{\g,G}(X)$ is isomorphic to an irreducible $G$-module $\tau$ such that
\[
\tau^{\n}\cong  X\otimes \wedge^{S} \overline{\n}.
\]

View $\C[G]$ as a $G\times G$-module such that the first factor acts through the left translation, and the second factor acts through the right translation.  View $\tau^\vee\otimes \tau$ as a $G\times G$-submodule  of $\C[G]$ via the linear embedding
\[
 \begin{array}{rcl}
 \tau^\vee\otimes \tau &\hookrightarrow &\C[G],\\
  \lambda\otimes u&\mapsto& (g\mapsto \la \lambda, g.u\ra ).
  \end{array}
\]
Since $ \oH^S(\g,L; \,\C[G]\otimes V)\cong \tau$ is irreducible, Peter-Weyl's Theorem implies that we have  an identification
\be\label{iden11}
 \oH^S(\g,L; \,\tau^\vee\otimes \tau \otimes V) = \oH^S(\g,L; \,\C[G]\otimes V).
\ee
\begin{lemt}\label{cycle}
There is an identification
\be\label{iden22}
\Hom_L(\wedge^S \g/\bar \q, (\tau^\vee)^{\bar \n}\otimes (\wedge^{2S} \g/\l)\otimes X)=\oH^S(\g,L; \,\tau^\vee\otimes V)
\ee
of one-dimensional vector spaces. \end{lemt}
\begin{proof}
This is known to experts. We sketch a proof for the convenience of the reader. Fix an element $x_0$ in the center of $\l$ such that all the eigenvalues of the operator $[x_0, \,\cdot\,] : \n\rightarrow \n$ are positive real numbers. By considering the action of $x_0$, we know that
\[
  \Hom_L(\wedge^j\g/\l, \tau^\vee\otimes V)=\left\{\begin{array}{ll}
                              \{0\},&\textrm{if }j\neq S,\\
                              \Hom_L(\wedge^S \g/\bar \q, (\tau^\vee)^{\bar \n}\otimes (\wedge^{2S} \g/\l)\otimes X),&\textrm{if }j= S.
                              \end{array}
                              \right.
\]
Hence the lemma follows.

\end{proof}

By \eqref{iden11} and \eqref{iden22}, we have  identifications
\begin{eqnarray}\label{iden33}
\widetilde \Pi_{\bar \q, L}^{\g,G}(X)&=&\Hom_L(\wedge^S \g/\bar \q, (\tau^\vee)^{\bar \n}\otimes (\wedge^{2S} \g/\l)\otimes X)\otimes \tau\\
\nonumber  &=&\Hom_L(\wedge^S \g/\bar \q, (\tau^\vee)^{\bar \n}\otimes (\wedge^{2S} \g/\l)\otimes X\otimes \tau)
\end{eqnarray}
of $G$-modules. Here we let $L$ act trivially on the $G$-module $\tau$.

\subsection{The linear functional $\widetilde \Pi_{\bar \q, L}^{\g,G}(\varphi)$}

Recall that
\[
  \varphi\in \Hom_{L\cap H}(X, \chi_{L\cap H}),
\]
where
\[
  \chi_{L\cap H}=\wedge^{2S}(\g/\l)^\vee\otimes \wedge^S \h/(\l\cap \h) \otimes \chi,
\]
to be viewed as an $(L\cap H)$-module.
By tensoring with $\wedge^{2S}\g/\l$, it obviously induces an $(L\cap H)$-equivariant linear map
\be\label{phip0}
   \varphi': (\wedge^{2S} \g/\l)\otimes X\rightarrow \wedge^S \h/(\l\cap \h) \otimes \chi.
\ee
Define a linear map
\[
  \begin{array}{rcl}
\widetilde \varphi: (\tau^\vee)^{\bar \n}\otimes (\wedge^{2S} \g/\l)\otimes X\otimes \tau&\rightarrow & \wedge^S \h/(\l\cap \h)\otimes \chi, \medskip\\
  \lambda \otimes \omega\otimes x \otimes u &\mapsto &\varphi'(\omega\otimes x ) \int_{H} \chi^{-1}(h) \,\la \lambda, h.u\ra \od\!h,
\end{array}
\]
where $\od\! h$ is the Haar measure on $H$ with total volume $1$.

Write
\be\label{idchi}
  \chi=\Hom_{L\cap H}(\wedge^S \h/(\l\cap \h), \wedge^S \h/(\l\cap \h)\otimes \chi).
\ee
Here we let $L\cap H$ act trivially on the $H$-module $\chi$.
Write
\[
  \iota_\h:  \wedge^S \h/(\l\cap \h)\rightarrow \wedge^S \g/\bar \q
\]
for the obvious linear map. This is a linear  isomorphism by  \eqref{qhg22}. The following lemma is routine to check.
\begin{lemt}\label{lemi}
Under the identifications \eqref{iden33} and \eqref{idchi}, the linear map
\[
 \widetilde \Pi_{\bar \q, L}^{\g,G}(\varphi): \widetilde \Pi_{\bar \q, L}^{\g,G}(X)\rightarrow \chi
\]
is identical to the following linear map:
\[
  \begin{array}{rcl}
  \Hom_L(\wedge^S \g/\bar \q, (\tau^\vee)^{\bar \n}\otimes (\wedge^{2S} \g/\l)\otimes X\otimes \tau)&\rightarrow & \Hom_{L\cap H}(\wedge^S \h/(\l\cap \h), \wedge^S \h/(\l\cap \h)\otimes \chi),\smallskip \\
 f &\mapsto & \widetilde \varphi \circ f\circ \iota_\h.
\end{array}
\]
\end{lemt}

Fix an isomorphism $X\cong \tau^{\n}\otimes \wedge^S \g/\bar \q$. Identify $\wedge^S \h/(\l\cap \h)$ with $\wedge^S \g/\bar \q$ via $\iota_\h$.  Then the linear map $\varphi'$ of \eqref{phip0} obviously induces an $(L\cap H)$-equivariant linear map
\[
   \varphi'': (\wedge^{2S} \g/\l)\otimes \tau^\n\rightarrow  \chi.
\]

It is easy to see that the latter map in Lemma \ref{lemi} is identical to the following linear map:
\[
    \begin{array}{rcl}
 ( (\tau^\vee)^{\bar \n}\otimes (\wedge^{2S} \g/\l)\otimes \tau^\n )^L\otimes \tau&\rightarrow &\chi,\smallskip \\
  \lambda \otimes \omega\otimes y \otimes u &\mapsto &\varphi''(\omega\otimes y) \int_{H} \chi^{-1}(h) \,\la \lambda, h.u\ra \od\!h.
\end{array}
\]
Here and as usual, a superscript group indicates the fixed vectors of the group action.

Note that $\wedge^{2S} \g/\l$ is a trivial $L$-module. Thus in order to finish the proof of Theorem \ref{mainp}, it suffices to prove the following proposition.
\begin{prpt}\label{prpsur}
The linear map
\[
\begin{array}{rcl}
  \Hom_{L\cap H}(\tau^\n, \chi)&\rightarrow &\Hom_H(((\tau^\vee)^{\bar \n}\otimes \tau^\n)^L\otimes \tau, \chi),\medskip \\
  \phi&\mapsto&\left(\lambda\otimes y\otimes u\mapsto \phi(y) \int_{H} \chi^{-1}(h) \,\la \lambda, h.u\ra \od\!h \right)
  \end{array}
\]
is surjective.
\end{prpt}

\subsection{A proof of Proposition \ref{prpsur}}

Write $\tau(\chi)$ for the $\chi$-isotypic subspace of $\tau$, and write $p_\chi: \tau\rightarrow \tau(\chi)$ for the $H$-equivariant projection map. Then the linear map in Proposition \ref{prpsur}  equals the following one:
\be\label{home1}
\begin{array}{rcl}
  \Hom_{L\cap H}(\tau^\n, \chi)&\rightarrow &\Hom_H(((\tau^\vee)^{\bar \n}\otimes \tau^\n)^L\otimes \tau, \chi)\medskip \\
  \phi&\mapsto&\left(\lambda\otimes y\otimes u\mapsto \phi(y) \la \lambda, p_\chi(u)\ra \right).
  \end{array}
\ee

Take a basis $(\lambda_1, \lambda_2, \cdots, \lambda_r)$ of $(\tau^\vee)^{\bar \n}$ and  $(y_1, y_2, \cdots, y_r)$ of $\tau^{ \n}$ which are dual to each other. Then  $((\tau^\vee)^{\bar \n}\otimes \tau^\n)^L$ is spanned by $\sum_{i=1}^r \lambda_i\otimes u_i$. Using this generator we identify $((\tau^\vee)^{\bar \n}\otimes \tau^\n)^L$ with $\C$ so that
\[
  ((\tau^\vee)^{\bar \n}\otimes \tau^\n)^L\otimes \tau=\tau.
\]
Then the linear map \eqref{home1} equals the composition of the linear map
\be\label{home2}
 \begin{array}{rcl}
  \Hom_{L\cap H}(\tau^\n, \chi)&\rightarrow &\chi\otimes (\tau^\vee)^{\bar \n}, \medskip \\
  \phi&\mapsto&\sum_{i=1}^r  \phi(y_i) \otimes \lambda_i
  \end{array}
\ee
and the linear map
\be\label{home3}
\begin{array}{rcl}
 \chi\otimes (\tau^\vee)^{\bar \n}&\rightarrow &\Hom_H(\tau, \chi), \\
 a\otimes \lambda &\mapsto&\left(u\mapsto a \la \lambda, p_\chi(u)\ra \right).
  \end{array}
\ee
Thus Proposition \ref{prpsur} follows from the following Lemma \ref{sur01} and Lemma \ref{sur02}.

\begin{lemt}\label{sur01}
The image of the linear map \eqref{home2} equals
 $(\chi\otimes (\tau^\vee)^{\bar \n})^{L\cap H}$.
\end{lemt}
\begin{proof}
Note that \eqref{home2} is a restriction of the obvious linear isomorphism
\[
\Hom_{\C}(\tau^\n, \chi)\xrightarrow{\,\cong\,} \chi\otimes (\tau^\vee)^{\bar \n}.
\]
This linear isomorphism is $H$-equivariant, and hence the lemma follows.
\end{proof}

\begin{lemt}\label{sur02}
The restriction of the linear map \eqref{home3} to $(\chi\otimes (\tau^\vee)^{\bar \n})^{L\cap H}$ is surjective.
\end{lemt}
\begin{proof}
By the $(L\cap H)$-invariance of the linear map  \eqref{home3},  it suffices to show that the map  \eqref{home3} itself is surjective.

Note that  $$\Hom_H(\tau, \chi)=\Hom_\C(\tau(\chi), \chi)=\chi\otimes (\tau( \chi))^\vee,$$ and   \eqref{home3} is identical to the following linear map:
\be\label{home4}
\begin{array}{rcl}
 \chi\otimes (\tau^\vee)^{\bar \n}&\rightarrow &\chi\otimes (\tau( \chi))^\vee, \\
 a\otimes \lambda &\mapsto& a\otimes \lambda|_{\tau( \chi)}.
  \end{array}
\ee

Using the first equality of \eqref{qhg22}, we have that
\begin{eqnarray*}
  \oU(\g). (\tau(\chi)\cap (\bar \n. \tau))&= & \oU(\bar \q). (\oU(\h). (\tau(\chi)\cap ( \bar \n. \tau)))\\
  &=& \oU(\bar \q).(\tau(\chi)\cap (\bar \n. \tau))\\
  &\subset & \bar \n. \tau\subsetneq \tau.
\end{eqnarray*}
Thus the irreducibility of $\tau$ implies that $\tau(\chi)\cap (\bar \n. \tau)=\{0\}$. Consequently, the natural map
\[
  \tau(\chi)\rightarrow \tau/(\bar \n. \tau)
\]
is injective. Taking the transpose, we know that the map
\[
\begin{array}{rcl}
(\tau^\vee)^{\bar \n}&\rightarrow &(\tau( \chi))^\vee, \\
  \lambda &\mapsto& \lambda|_{\tau( \chi)}
    \end{array}
\]
is surjective. Thus the map \eqref{home4} is also surjective and the lemma is proved.

\end{proof}

\section{The first application}\label{firsta}

In this section and the next one, we retain the notation of the Introduction.
 As applications of our construction, we give two arithmetically interesting examples of cohomological test vectors.  For the first one, suppose $(G, H, \theta, \sigma)$ and  $\K$ are as in Example \ref{exm12} so that
 \[
 G=\GL_n(\K), \quad H=\GL_n(\R), \quad K=\oU(n)\quad \textrm{and}\quad  H_\rc=\oO(n).
 \] 
Write $\iota_1, \iota_2:\K\rightarrow \C$ for the two distinct isomorphisms.

\subsection{Cohomological test vectors}\label{secct1}

 Fix a sequence
\be\label{mu1}
  \mu=(\mu_1\geq \mu_2\geq \cdots \geq \mu_n;\, \mu_{n+1}\geq \mu_{n+2}\geq \cdots \geq \mu_{2n} )\in \Z^{2n}
\ee
such that
\begin{equation}
  \mu_{1}+\mu_{2n}=\mu_{2}+\mu_{2n-1}=\cdots=\mu_{n}+\mu_{n+1}=0.
\end{equation}

Write $\kappa: \K^\times\rightarrow  \C^\times$ for the unitary character whose square equals $\frac{(\iota_1)|_{\K^\times}}{(\iota_2)|_{\K^\times}}$. 
Define the principal series representation
\be\label{defind}
  \pi_\mu:=\Ind_{\mathrm B_n(\K)}^{\GL_n(\K)}\left(\kappa^{2\mu_1+n-1}\otimes \kappa^{2\mu_2+n-3}\otimes \cdots \otimes \kappa^{2\mu_n+1-n}\right),
\ee
where $\mathrm B_n(\K)$ denotes the Borel subgroup of the upper-triangular matrices. This is an irreducible Casselman-Wallach representation of $\GL_n(\K)$ (\cite{Vog0}, see also \cite[Theorem 1.2]{Bar}).

Denote by $F_\mu$ the irreducible algebraic representation of $\GL_n(\C)\times \GL_n(\C)$ of highest weight $\mu$. It is also viewed as an irreducible representation of  $\GL_n(\K)$ by restricting through the complexification map
\begin{equation}\label{compl}
  \GL_n(\K)\rightarrow \GL_n(\C)\times \GL_n(\C),\quad g\mapsto (\iota_1(g), \iota_2( g)).
\end{equation}

By Vogan-Zuckerman theory of cohomological representations \cite{VZ},  $\pi_\mu$  is the unique (up to isomorphism) irreducible Casselman-Wallach representation  of $\GL_n(\K)$ which is unitarizable and tempered, and whose total relative Lie algebra cohomology is nonzero (see \cite[Section 3]{Clo}):
                           \begin{equation}\label{cohome}
                             \oH^*(\gl_{n}(\C)\times \gl_n(\C),\mathrm{U}(n);
                           F_\mu^\vee \otimes \pi_\mu)\neq 0.
                            \end{equation}
Here  $\gl_{n}(\C)\times \gl_n(\C)$ is identified with $\g$  through the differential of \eqref{compl}.

Recall that a result of Vogan \cite[Theorem 4.9]{Vog0} asserts that every irreducible Casselman-Wallach representation of $\GL_n(\K)$ has a unique minimal $\oU(n)$-type, and it occurs with multiplicity one.
Likewise, every irreducible Casselman-Wallach representation of $\GL_n(\R)$ has a unique minimal $\oO(n)$-type, and it occurs with multiplicity one.

Whenever a Lie group has exactly two connected components, we use $\sgn$ to denote the unique non-trivial quadratic character of it. We will prove the following Theorem in the next subsection.

\begin{thm}\label{test1}
The space $\Hom_{\GL_n(\R)}(\pi_\mu,\sgn^{n-1})$ is one-dimensional, and a nonzero element of it does not vanish on the minimal $\oU(n)$-type of $\pi_\mu$.

\end{thm}

\begin{remarks}
(a) Using the explicit description of $\pi_\mu$ in \eqref{defind}, the analytic method of meromorphic continuations also produces a nonzero element of $\Hom_{\GL_n(\R)}(\pi_\mu,\sgn^{n-1})$ (\cf \cite[Theorem 6]{Osh}, \cite[Theorem 5.1]{Ola} and \cite[Theorem 5.10]{Ban}). But it  seems hard to show the non-vanishing on the minimal $\oU(n)$-type  by this analytic method.

(b)  When $n\geq 2$, some matrix coefficients of $\pi_\mu$ are not integrable on $\SL_n(\R)$. Thus the method of \cite[Lemma 4.5]{Sun}, which proves a similar result in the Gross-Prasad case by matrix coefficient integrals, does not apply to Theorem \ref{test1}.

\end{remarks}

\subsection{A proof of Theorem \ref{test1}}\label{pfirst}
Note that $\g=\gl_n(\C)\times \gl_n(\C)$ and the complexified differential of $\theta: G\rightarrow G$ equals
\[
  \theta: \gl_n(\C)\times \gl_n(\C)\rightarrow \gl_n(\C)\times \gl_n(\C),\quad (x,y)\mapsto (-y^{\mathrm t}, -x^{\mathrm t}).
\]
Suppose  \[
\q=\b_n\times \b_n^\mathrm t\subset \g=\gl_n(\C)\times \gl_n(\C),
 \]
where $\b_n$ denotes the Lie algebra of all upper-triangular matrices in $\gl_n(\C)$. Then $\q$ is a $\theta$-stable parabolic subalgebra of $\g$ satisfying \eqref{q}, and
\[
   (L,\,L_\rc,\,L\cap H,\,L_\rc\cap H_\rc)=((\K^\times)^n,\, (\oU(1))^n,\, (\R^\times)^n,\,\{\pm 1\}^n).
\]

Using the isomorphisms $\iota_1, \iota_2: \K\rightarrow \C$, we also have an identification $$\l=\C^n\times \C^n.$$
 Denote by $\lambda_\mu$ the unitary character of $L$ whose complexified differential equals
\[
   (\mu_1+n-1, \mu_2+n-3,\cdots, \mu_n+1-n; \mu_{2n}+1-n, \mu_{2n-1}+3-n, \cdots, \mu_{n+1}+n-1).
\]
Then by Vogan-Zuckerman theory \cite{VZ}, $\Pi_{\bar \q, L_\rc}^{\g,K}(\lambda_\mu)$ is isomorphic to the $(\g,K)$-module of $K$-finite vectors in $\pi_\mu$, and the irreducible representation $\tau_\mu:=\Pi_{\overline{\q_\mathrm c}, L_\rc}^{\k,K}(\lambda_\mu)$ of $K$ occurs with multiplicity one in $\pi_\mu$ (it is the unique minimal $K$-type of $\pi_\mu$). Identify $\K$ with $\C$ via $\iota_1$, then $\tau_\mu$ has highest weight
\[
  (2\mu_1+n-1, 2\mu_2+n-3, \cdots, 2\mu_n+1-n).
\]

\begin{lemt}\label{hegc10}
One has that
\[
\dim \Hom_{H_\rc}(\tau_\mu, \sgn^{n-1})=1.
\]
\end{lemt}
\begin{proof}
This is an instance of Cartan-Helgason Theorem (\cf \cite[Chapter V, Theorem 4.1]{Hel}).
\end{proof}

\begin{lemt}\label{existphi01}
There is an element of $\Hom_{\h,H_\rc}(\Pi_{\bar \q, L_\rc}^{\g,K}(\lambda_\mu), \sgn^{n-1})$ which does not vanish on the minimal $K$-type $\tau_\mu$ of $\Pi_{\bar \q, L_\rc}^{\g,K}(\lambda_\mu)$.
\end{lemt}
\begin{proof}
Note that the one-dimensional representations
\[
  (\lambda_\mu)|_{L\cap H}\quad \textrm{and}\quad \chi_{L\cap H}\cong \sgn^{n-1}|_{L\cap H}\otimes \wedge^{\mathrm{top}} (\h_\rc /(\l_\rc\cap \h_\rc))
\]
of $L\cap H$ are both trivial. Thus Lemma \ref{hegc10} implies that
\[
 \dim \Hom_{ L_\rc\cap H_\rc }(\lambda_\mu,\chi_{L\cap H})=\dim \Hom_{H_\rc}(\tau_\mu, \sgn^{n-1}).
\]
Since $$
 \Hom_{\l\cap \h, L_\rc\cap H_\rc}(\lambda_\mu, \chi_{L\cap H})\neq \{0\},
$$
 the lemma follows by Theorem \ref{mainf}.
\end{proof}
\begin{lemt}\label{mulone1}
For every irreducible Casselman-Wallach representation $\pi$ of $\GL_n(\K)$, and every character $\chi_0$ of $\GL_n(\R)$, one has that
\[
  \dim \Hom_{\GL_n(\R)}(\pi, \chi_0)\leq 1.
\]
\end{lemt}
\begin{proof}
When $\chi_0$ is trivial, the lemma is proved in \cite[Theorem 8.2.5]{AGS}.  In general, since $\chi_0$ extends to a character of $\GL_n(\K)$, the lemma reduces to the case when $\chi_0$ is trivial.
\end{proof}

Now Theorem \ref{test1} follows by combining Theorem \ref{autcon0}, Lemma \ref{existphi01} and Lemma \ref{mulone1}.

\subsection{The  non-vanishing hypothesis for Asai L-functions}

Write
\[
  \GL_n(\K)= \GL_n^1(\K) \times \R_+^\times,
\]
where $\GL^1_n(\K):=\{g\in\GL_n(\K)\mid \iota_1(\det g)\cdot \iota_2(\det g)=1\}$, and the group $ \R_+^\times$ of positive real numbers is viewed as a subgroup of $\GL_n(\K)$ via the diagonal embedding. Following \cite{GHL}, write $\m_G$ for the complexified Lie algebra of $\GL^1_n(\K)$.


Put $t_n:=\frac{(n-1)(n+2)}{2}$. Recall that  \cite[Lemma 3.14]{Clo}
\[
\dim \oH^{t_n}(\m_G,\mathrm{U}(n);
                           F_\mu^\vee \otimes \pi_\mu)=1,
\]
and
\[
  \oH^{j}(\m_G,\mathrm{U}(n);
                           F_\mu^\vee \otimes \pi_\mu)=\{0\}\qquad \textrm{for all $j>t_n$}.
\]
Note that
\[
  \dim \Hom_{\GL_n(\R)}(F_\mu^\vee,\C)=1 \quad\textrm{and}\quad \dim \oH^{t_n}(\sl_{n}(\C),\mathrm{O}(n);
                           \sgn^{n-1})=1.
                           \]

\begin{thm}\label{nv1}
Let $\varphi$ be a nonzero element of $\Hom_{\GL_n(\R)}(\pi_\mu,\sgn^{n-1})$, and let $\psi$ be a nonzero element of $\Hom_{\GL_n(\R)}(F_\mu^\vee,\C)$. Then by restriction of cohomology, the linear functional $\psi\otimes \varphi: F_\mu^\vee \otimes \pi_\mu\rightarrow \sgn^{n-1}$ induces a nonzero linear map
\[
   \oH^{t_n}(\m_G,\mathrm{U}(n);
                           F_\mu^\vee \otimes \pi_\mu)\rightarrow  \oH^{t_n}(\sl_{n}(\C),\mathrm{O}(n);
                           \sgn^{n-1})
\]
of one-dimensional vector spaces.
\end{thm}

Theorem \ref{nv1} is a representation theoretic reformulation of the non-vanishing hypothesis of Grobner-Harris-Lapid in the study of non-critical values of the Asai L-function (see \cite[Section 6.2]{GHL}).

In view of Theorem \ref{test1}, Theorem \ref{nv1}  follows by applying Theorem \ref{nonvmod} of Appendix \ref{appendix} to the group $\GL_n^1(\K)$.

\section{The second application}\label{seconda}

For the second application, suppose $(G, H, \theta, \sigma)$ is as in Example \ref{exm13} so that
\[
 G=\GL_{2n}(\R), \quad H=\GL_n(\R)\times \GL_n(\R), \quad K=\oO(2n)\quad \textrm{and}\quad  H_\rc=\oO(n)\times \oO(n).
 \]

\subsection{Cohomological test vectors}
Fix a sequence
\[
  \nu=(\nu_1\geq \nu_2\geq \cdots \geq \nu_{2n} )\in \Z^{2n}\qquad (n\geq 1)
\]
such that
\begin{equation}\label{w}
  \nu_{1}+\nu_{2n}=\nu_{2}+\nu_{2n-1}=\cdots=\nu_{n}+\nu_{n+1}=w
\end{equation}
for some $w\in \Z$.
Define
\[
  \pi^\nu:=\abs{\det}^{\frac{w}{2} }\otimes \Ind_{P_{2,2,\cdots, 2}}^{G} \left(D_{\nu_1-\nu_{2n}+2n-1} \widehat{\otimes} D_{\nu_2-\nu_{2n-1}+2n-3} \widehat{\otimes} \cdots \widehat{\otimes} D_{\nu_n-\nu_{n+1}+1} \right),
\]
where   $D_k$ ($k\geq 1$) is a relative discrete series representation as in Example \ref{exa110}.  Similar to \eqref{defind}, $\pi^\nu$ is an irreducible Casselman-Wallach representation of $\GL_{2n}(\R)$.

Denote by $F^\nu$ the irreducible algebraic representation of $\GL_{2n}(\C)$ of highest weight $\nu$. It is also viewed as an irreducible representation of $\GL_{2n}(\R)$ by restriction. Similar to Section \ref{secct1}, $\pi^\nu$ is the unique  (up to isomorphism) irreducible Casselman-Wallach representation  of $\GL_{2n}(\R)$ such that \cite[Section 3]{Clo}
\begin{itemize}
                           \item
                            $\pi^\nu|_{{\operatorname{SL}_{2n}^{\pm}}(\R)}$ is unitarizable and
                            tempered, and
                           \item the total relative Lie algebra cohomology
                           \begin{equation}\label{cohom}
                             \oH^*(\gl_{2n}(\C),\SO(2n);
                           (F^\nu)^\vee\otimes \pi^\nu)\neq 0,
                           \end{equation}
                             \end{itemize}
where
\[
  {\operatorname{SL}_{2n}^{\pm}}(\R):=\{g\in \GL_{2n}(\R)\mid \det(g)=\pm 1\}.
  \]

 Suppose
\be\label{defchi12}
  \chi=\chi_1\otimes \chi_2: \GL_n(\R)\times \GL_n(\R)\rightarrow \C^\times
  \ee
   is a character such that
\[
 \chi_1\cdot \chi_2={\det}^{w}.
  \]
 For each $s\in \C$, let $\abs{\det}^{s,-s}$ denotes the character $\abs{\det}^s\otimes \abs{\det}^{-s}$ of $\GL_n(\R)\times \GL_n(\R)$.
The following Theorem will be proved in the next subsection.


\begin{thm}\label{test2}
 Up to scalar multiplication, there exists a unique nonzero element $\varphi\in \Hom_{\GL_n(\R)\times \GL_n(\R)}(\pi^\nu,\chi)$ which extends to a holomorphic family in the following sense: there exists a map
 \[
  \zeta:  \pi^\nu\times \C\rightarrow \C
 \]
 such that
 \begin{itemize}
   \item $\zeta(\,\cdot\,, s)\in \Hom_{\GL_n(\R)\times \GL_n(\R)}(\pi^\nu, \chi\cdot \abs{\det}^{s,-s})$,  for all $s\in \C$;
   \item $\zeta(v,\,\cdot\,)$ is an entire function, for all $\oO(2n)$-finite vector $v\in \pi^\nu$;
   \item $\zeta(\,\cdot\,,0)=\varphi$.
 \end{itemize}
 Moreover, $\varphi$ does not vanish on the minimal $\oO(2n)$-type of $\pi^\nu$.
 \end{thm}

Similar to Theorem \ref{test1}, a nonzero $\chi$-invariant continuous linear functional on $\pi^\nu$ may also be constructed by the analytic method of meromorphic continuations. But the non-vanishing on the minimal $\oO(2n)$-type is not easy to show by that method.

\subsection{A proof of Theorem \ref{test2}}


Fix an embedding
\begin{equation}\label{gamma}
  \gamma_{2n}: (\C^{\times})^{2n}\hookrightarrow \GL_{2n}(\C)
\end{equation}
of algebraic groups which sends $(a_1,a_2,\cdots,a_{2n})$ to the
matrix
\[
\left[
  \begin{array}{cccccccc}
    \frac{a_1+a_{2n}}{2} & 0 & \cdots & 0 & 0 &\cdots & 0 & \frac{a_1-a_{2n}}{2\mathbf i}  \\
     0&\frac{a_2+a_{2n-1}}{2}  & \cdots & 0 & 0 &\cdots & \frac{a_2-a_{2n-1}}{2\mathbf i}  & 0 \\
     \cdots & \cdots & \cdots & \cdots & \cdots &\cdots & \cdots &\cdots  \\
   0 & 0 & \cdots & \frac{a_n+a_{n+1}}{2} & \frac{a_n-a_{n+1}}{2\mathbf i} &\cdots & 0 & 0  \\
   0 & 0 & \cdots & \frac{a_{n+1}-a_n}{2\mathbf i}  & \frac{a_{n+1}+a_n}{2} &\cdots & 0 & 0  \\
  \cdots & \cdots & \cdots & \cdots & \cdots &\cdots & \cdots &\cdots  \\
   0&\frac{a_{2n-1}-a_2}{2\mathbf i}  & \cdots & 0 & 0 &\cdots & \frac{a_{2n-1}+a_2}{2}  & 0 \\
   \frac{a_{2n}-a_1}{2\mathbf i} & 0 & \cdots & 0 & 0 &\cdots &  & \frac{a_{2n}+a_1}{2}  \\
  \end{array}
\right],
\]
where $\mathbf i=\sqrt{-1}\in \C$ is the fixed square root of $-1$.
View $(\C^\times)^{2n}$ as a Cartan subgroup of $\GL_{2n}(\C)$ via
the embedding \eqref{gamma}. Then the corresponding root system is
\begin{equation}\label{phi2n0}
  \{\pm (e_i-e_j)\mid 1\leq i<j\leq 2n\}\subset
\Z^{2n}.
\end{equation}
Here $e_1, e_2,\cdots, e_{2n}$ denote the standard basis of $\Z^{2n}$.
Suppose $\q$ is the Borel subalgebra  of $\g$ which corresponds to
the positive system
\begin{equation}\label{phi2np}
  \{e_i-e_j\mid 1\leq i<j\leq 2n\}\subset
\Z^{2n}
\end{equation}
of \eqref{phi2n0}. Then $\q$ is a $\theta$-stable parabolic subalgebra of $\g$ satisfying \eqref{q}, and
\[
  (L,\,L_\rc,\,L\cap H,\,L_\rc\cap H_\rc)=((\C^\times)^n, \,(\mathbb{S}^1)^n,\,(\R^\times)^n,\,\{\pm 1\}^n),
\]
where $\mathbb S^1$ denotes the group of complex numbers of modulus one. Here $L$ is viewed as a subgroup of $G$ via the embedding
\[
   \begin{array}{rcl}
   (\C^\times)^n&\rightarrow& \GL_{2n}(\R),\\
   (a_1,a_2,\cdots, a_n)&\mapsto&\gamma_{2n}(a_1,a_2,\cdots, a_n,\bar a_n,\cdots,  \bar
  a_2,\bar a_1).
  \end{array}
\]

 Denote by $\lambda^\nu$ the restriction to $L$ of the character
\[
  (\nu_1+2n-1, \nu_2+2n-3, \cdots, \nu_{2n}+1-2n)
\]
of $(\C^\times)^{2n}$, through the embedding
\[
   \begin{array}{rcl}
   L=(\C^\times)^n&\rightarrow& (\C^\times)^{2n},\\
   (a_1,a_2,\cdots, a_n)&\mapsto&(a_1,a_2,\cdots, a_n,\bar a_n,\cdots,  \bar
  a_2,\bar a_1).
  \end{array}
\]
Then by Vogan-Zuckerman theory \cite{VZ}, we know that $\Pi_{\bar \q, L_\rc}^{\g,K}(\lambda^\nu)$ is isomorphic to the $(\g,K)$-module of $K$-finite vectors in $\pi^\nu$.  Moreover, the irreducible $K^\circ$-module $\tau^\nu:=\Pi_{\overline{\q_\mathrm c}, L_\rc}^{\k,K^\circ}(\lambda^\nu)$ occurs with multiplicity one in $\pi^\nu$ (it is contained in the unique minimal $K$-type of $\pi^\nu$).

Note that $L_\rc$ is a Cartan subgroup of $K^\circ$, and $\q_\mathrm c$ is a Borel subalgebra of $\k$ which corresponds to the positive system
\[
  \{e_i\pm e_j\mid 1\leq i<j\leq n\}\subset \Z^n=\Hom(L_\rc, \C^\times)
\]
of the root system of $K^\circ$. The highest weight of $\tau^\nu$ is
\[
  (\nu_1-\nu_{2n}+2n, \nu_2-\nu_{2n-1}+2(n-1),\cdots, \nu_n-\nu_{n+1}+2).
\]

Note that $H\cap K^\circ$ equals
\[
  \oO(n)\times_{\{\pm 1\}} \oO(n)\quad (\textrm{the fiber product over the  determinant homomorphisms}).
  \]
Thus it  has exactly two connected components. Recall the integer $w$ from \eqref{w}.
  \begin{lemt}\label{hc2}
  One has that
  \[
    \dim \Hom_{H\cap K^\circ}(\tau^\nu, \sgn^w)=1.
  \]
  \end{lemt}
  \begin{proof}
This is also an instance of Cartan-Helgason Theorem (\cf \cite[Chapter V, Theorem 4.1]{Hel}).
\end{proof}

\begin{lemt}\label{existphi1}
There is a map
 \be\label{zetapi}
  \zeta_0:  \Pi_{\bar \q, L_\rc}^{\g,K}(\lambda^\nu)\times \C\rightarrow \C
 \ee
 with the following properties:
  \begin{itemize}
   \item $\zeta_0(\,\cdot\,, s)\in \Hom_{\h,H_\rc}(\Pi_{\bar \q, L_\rc}^{\g,K}(\lambda^\nu), \chi\cdot \abs{\det}^{s,-s})$,  for all $s\in \C$;
   \item $\zeta_0(v,\,\cdot\,)$ is a polynomial function, for all $v\in \Pi_{\bar \q, L_\rc}^{\g,K}(\lambda^\nu)$;
   \item $\zeta_0(\,\cdot\, , s)$ does not vanish on $\tau^\nu\subset \Pi_{\bar \q, L_\rc}^{\g,K}(\lambda^\nu)$, for all $s\in \C$.
 \end{itemize}
\end{lemt}
\begin{proof}
Note that the $(L_\rc\cap H_\rc)$-action on $\wedge^{S} \h_\rc/(\l_\rc \cap \h_\rc)$ is trivial, and  the character $\lambda^\nu|_{L\cap H}$ and $(\chi\cdot \abs{\det}^{s,-s})|_{L\cap H}$ ($s\in \C$) are both equals to the following one:
\[
  (a_1, a_2, \cdots, a_n)\mapsto (a_1 a_2 \cdot\cdots \cdot a_n)^w.
\]

Fix a nonzero element
\[
  \varphi\in \Hom_{L\cap H}(\lambda^\nu, \wedge^{2S}(\k/\l_\rc)^\vee\otimes \wedge^S \h_\rc/(\l_\rc \cap \h_\rc) \otimes \chi).
\]
For all $s\in \C$, let
\[
  \varphi_s\in \Hom_{L\cap H}(\lambda^\nu, \wedge^{2S}(\k/\l_\rc)^\vee\otimes \wedge^S \h_\rc/(\l_\rc \cap \h_\rc) \otimes (\chi\cdot \abs{\det}^{s,-s}))
\]
be the element which is identical to $\varphi$ when both $\chi$ and $\chi\cdot \abs{\det}^{s,-s}$ are identified with $\C$ as vector spaces.
Now we define a map
 \[
  \zeta_0:  \Pi_{\bar \q, L_\rc}^{\g,K}(\lambda^\nu)\times \C\rightarrow \C, \quad (v,s)\mapsto (\Pi_{\bar \q, L_\rc}^{\g,K}(\varphi_s))(v).
 \]
Then the first two properties of the lemma hold by the construction of $\Pi_{\bar \q, L_\rc}^{\g,K}(\varphi_s)$.
The third property holds by Proposition \ref{commb0}, Lemma \ref{hc2}, and Theorem \ref{main}.
\end{proof}

Identify $\pi^\nu$ with the Casselman-Wallach globalization of $\Pi_{\bar \q, L_\rc}^{\g,K}(\lambda^\nu)$. By the automatic continuity theorem (Theorem \ref{autcon0}), the map $\zeta_0$ of \eqref{zetapi} extends to a map
\[
  \zeta: \pi^\nu\times \C\rightarrow \C
\]
such that
 \begin{itemize}\label{item3z}
   \item $\zeta(\,\cdot\,, s)\in \Hom_{\GL_n(\R)\times \GL_n(\R)}(\pi^\nu, \chi\cdot \abs{\det}^{s,-s})$,  for all $s\in \C$;
   \item $\zeta(v,\,\cdot\,)$ is a polynomial function, for all $\oO(2n)$-finite vectors $v\in \pi^\nu$;
   \item $\zeta(\,\cdot\,,s)$ does not vanish on $\tau^\nu\subset \Pi_{\bar \q, L_\rc}^{\g,K}(\lambda^\nu)$, for all $s\in \C$.
 \end{itemize}
Then $\varphi:=\zeta(\,\cdot\, , 0)$ satisfies the conditions of Theorem \ref{test2}. Moreover, $\varphi$ does not vanish on the minimal $K$-type of $\pi^\nu$.

To prove the uniqueness of $\varphi$, recall the following multiplicity one result.

\begin{lemt}\label{cs}(\cite[Theorem B]{CS})
Let $\pi$ be an irreducible Casselman-Wallach representation of $\GL_{2n}(\R)$. Then for all but countably many characters $\chi'$ of $\GL_n(\R)\times \GL_n(\R)$, the space $\Hom_{\GL_n(\R)\times \GL_n(\R)}(\pi, \chi')$ is at most one-dimensional.
\end{lemt}

Now let $\zeta': \pi^\nu\times \C\rightarrow \C$ be a map such that
 \begin{itemize}\label{item3z}
   \item $\zeta'(\,\cdot\,, s)\in \Hom_{\GL_n(\R)\times \GL_n(\R)}(\pi^\nu, \chi\cdot \abs{\det}^{s,-s})$,  for all $s\in \C$;
   \item $\zeta'(v,\,\cdot\,)$ is an entire function, for all $\oO(2n)$-finite vectors $v\in \pi^\nu$.
   \end{itemize}
Pick a vector $v_0\in \tau^\nu\subset \pi^\nu$ which does not vanish under a nonzero element of $\Hom_{H\cap K^\circ}(\tau^\nu, \sgn^w)$. Then $\zeta(v_0, \cdot)$ is a nowhere vanishing polynomial  function, and is thus a nonzero constant. Put
\[
  \gamma(s):=\frac{\zeta'(v_0, s)}{\zeta(v_0, s)},\quad s\in \C.
\]
Then Lemma \ref{cs} implies that
\be\label{allc1}
   \zeta'(\cdot, s)=\gamma(s)\zeta(\cdot, s),
\ee
for all but countably many $s\in \C$. Therefore for all $v\in \pi^\nu$ which is $K$-finite, the continuity of the both sides of \eqref{allc1} on the variable $s\in \C$ implies that
\be\label{allc2}
   \zeta'(v, s)=\gamma(s)\zeta(v, s)\quad \textrm{for all $s\in \C$.}
\ee
Finally, the continuity of the both sides of the equality \eqref{allc2} on the variable $v\in \pi^\nu$ implies that
\be\label{allc3}
   \zeta'(\cdot, s)=\gamma(s)\zeta(\cdot, s)\quad \textrm{for all $s\in \C$.}
\ee
In particular, $\zeta'(\cdot, 0)$ is a scalar multiple of $\varphi=\zeta(\cdot, 0)$. This proves the uniqueness of $\varphi$, and finishes the proof of Theorem  \ref{test2}.

\subsection{The  non-vanishing hypothesis for L-functions for $\mathrm{GSpin}(2n+1)$}

Put $t_n':=n^2+n-1$. Then  \cite[Lemma 3.14]{Clo}
\[
 \oH^\nu:= \oH^{t_n'}(\sl_{2n}(\C),\mathrm{SO}(2n);
                           (F^\nu)^\vee\otimes \pi^\nu)\neq 0,
\]
and
\[
  \oH^{j}(\sl_{2n}(\C),\mathrm{SO}(2n);
                           (F^\nu)^\vee\otimes \pi^\nu)=0\qquad \textrm{for all $j>t_n'$}.
\]
The natural actions of the group $\oO(2n)$ on $\sl_{2n}(\C)$, $\mathrm{SO}(2n)$ and $(F^\nu)^\vee\otimes \pi^\nu$ induce a representation of $\oO(2n)/\SO(2n)=\{\pm 1\}$ on $\oH^\nu$. It turns out that \cite[Equation (3.2)]{Mah}
\be\label{ohnu}
  \oH^\nu\cong \C\oplus \sgn \qquad(\textrm{as representations of  the group $\{\pm 1\}$}),
\ee
where ``$\,\C\,$" stands for the trivial representation.

Recall from \eqref{defchi12} the character  $\chi=\chi_1\otimes \chi_2$ with $\chi_1\cdot \chi_2={\det}^{w}$.  Now we assume that
\be\label{rglp}
  (\chi_i)|_{(\GL_n(\R))^\circ}=({\det}^{w_i})|_{(\GL_n(\R))^\circ}, \qquad i=1,2,
\ee
for two integers $w_1,w_2$  such that
\be\label{w12}
  w_1+w_2=w \quad \textrm{and} \quad \nu_n\geq w_i\geq \nu_{n+1} \quad (i=1,2).
\ee
 Then the space
\[
 \oH_\chi:=\oH^{t_n'}(\sl_{2n}(\C)\cap(\gl_{n}(\C)\times \gl_n(\C)), \oO(n)\times_{\{\pm 1\}} \oO(n);
                           {\det}^{-w_1,-w_2}\cdot\chi)
\]
is one-dimensional, and naturally carries a representation of
\[
(\oO(n)\times \oO(n))
/(\oO(n)\times_{\{\pm 1\}} \oO(n))=\{\pm 1\}.
\]
Here  ${\det}^{-w_1,-w_2}$ denotes the character ${\det}^{-w_1}\otimes {\det}^{-w_2}$ of $\GL_n(\C)\times \GL_n(\C)$.
Using  \eqref{ohnu}, we conclude that
\[
  \dim \Hom_{\{\pm 1\}} (\oH^\nu, \oH_\chi)=1.
\]

As an instance of H. Schlichtkrull's
generalization of Cartan-Helgason Theorem (\cite[Theorem 7.2]{Sch}, see also \cite[Theorem 2.1]{Kna}), \eqref{w12} implies that
\[
 \dim \Hom_{\GL_n(\C)\times \GL_n(\C)}((F^\nu)^\vee,
  {\det}^{-w_1,-w_2})=1.
\]


\begin{thm}\label{nv2}
Assume that  \eqref{rglp} and \eqref{w12}  hold. Let $\varphi$ be as in Theorem \ref{test2}. Let $\psi$ be a nonzero element of $\Hom_{\GL_n(\C)\times \GL_n(\C)}((F^\nu)^\vee,
  {\det}^{-w_1,-w_2})$. Then by restriction of cohomology, the linear functional $\psi\otimes \varphi: (F^\nu)^\vee\otimes\pi^\nu\rightarrow  {\det}^{-w_1,-w_2}\cdot \chi$ induces a nonzero element of $\Hom_{\{\pm 1\}} (\oH^\nu, \oH_\chi)$.
\end{thm}

Theorem \ref{nv2} is a representation theoretic reformulation of the non-vanishing hypothesis in the study of  $p$-adic L-functions and critical values of L-functions for $\mathrm{GSpin}(2n+1)$, using the Langlands lift to $\GL(2n)$ and Shalika models (see \cite[assumption (A2)]{AG}, \cite[Section 6.6]{GR} and \cite{AS}).

Now we come to the proof of  Theorem \ref{nv2}. Denote by $\pi^\nu_\circ$ the Casselman-Wallach globalization of $\Pi_{\bar \q, L_\rc}^{\g,K^\circ}(\lambda^\nu)$. It is an irreducible representation of $(\GL_{2n}(\R))^\circ$ with minimal $K^\circ$-type $\tau^\nu$. Moreover, we have a natural inclusion
$\pi^\nu_\circ\subset \pi^\nu$, and
\[
  \oH^\nu_\circ:=\oH^{t_n'}(\sl_{2n}(\C),\mathrm{SO}(2n);
                           (F^\nu)^\vee\otimes  \pi^\nu_\circ)
\]
 is a one-dimensional subspace of the two-dimensional space
\[
  \oH^\nu:=\oH^{t_n'}(\sl_{2n}(\C),\mathrm{SO}(2n);
                           (F^\nu)^\vee\otimes \pi^\nu).
\]

Denote by $\varphi_\circ$ the restriction of $\varphi$ to $\pi^\nu_\circ$, which does not vanish on the minimal $K^\circ$-type $\tau^\nu$.
Let $\psi: (F^\nu)^\vee\rightarrow {\det}^{-w_1,-w_2}$ be as in Theorem \ref{nv2}. By restriction of cohomology, the linear functionals
\[
\psi\otimes \varphi_\circ: (F^\nu)^\vee\otimes\pi^\nu_\circ\rightarrow  {\det}^{-w_1,-w_2}\cdot \chi\quad\textrm{and}\quad \psi\otimes \varphi: (F^\nu)^\vee\otimes\pi^\nu\rightarrow  {\det}^{-w_1,-w_2}\cdot \chi
\]
 respectively induce linear functionals
\[
\eta_\circ:  \oH^\nu_\circ\rightarrow \oH_\chi\quad\textrm{and}\quad \eta:  \oH^\nu\rightarrow \oH_\chi.
\]
Applying Theorem \ref{nonvmod} of Appendix \ref{appendix} to the group $\SL_{2n}(\R)$, we know that $\eta_\circ$ is nonzero. Therefore $\eta$ is nonzero since it extends $\eta_\circ$. This finishes the proof of  Theorem \ref{nv2}.

\appendix

\section{Modular symbols at infinity}\label{appendix}

As mentioned in the Introduction, the Archimedean behaviors of modular symbols are captured by certain restriction maps of relative Lie algebra cohomologies.  We call these  restriction maps modular symbols at infinity.
In this appendix, we show (in Theorem \ref{nonvmod}) that the non-vanishing on the bottom layers implies the non-vanishing of certain  modular symbols at infinity.

\subsection{Unitary representations with nonzero cohomology}\label{unzero}

We first review some basic facts concerning unitary representations with nonzero cohomology.
Let $G$ be a real reductive group (as in Section \ref{genrality}). For simplicity, assume it is connected and  its center is compact.  Fix a Cartan  involution $\theta$ of $G$. Put $K:=G^\theta$, which is a maximal compact subgroup of $G$.

As usual, we use the corresponding lower case Gothic letter to indicate the complexified
Lie algebra of a Lie group.  Let $F$ be an irreducible finite-dimensional representation of $G$, and let $V$ be a unitarizable  irreducible $(\g,K)$-module such that the total relative Lie algebra cohomology
\[
   \oH^*(\g, K; F^\vee\otimes V)\neq 0.
\]
Then by \cite{VZ}, there is a $\theta$-stable parabolic subalgebra $\q$ of $\g$ with the following properties:
\begin{itemize}
  \item the representation $F^\n|_{L}$ is one-dimensional and unitarizable;
  \item $V\cong \Pi_S( F^\n\otimes \wedge^{\dim\n}\n)$.
  \end{itemize}
Here
\[
  \left\{
    \begin{array}{l}
      \n:=\textrm{the nilpotent radical of }\q\cap [\g,\g], \\
      L:=\RN_G(\q)=\RN_G(\bar \q),\\
       S:=\dim (\n\cap \k),\\
    \end{array}
  \right.
\]
and  $\Pi_S$ denotes the $S$-th left derived functor of the functor
\[
  \oR(\g,K)\otimes_{\oR(\bar \q, {L_\rc})} (\,\cdot\,)\qquad (L_\rc:=L\cap K)
\]
from the category of $(\bar \q,{L_\rc})$-modules to the category of $(\g,K)$-modules.

Put
 \[
\q_\mathrm c:=\q\cap \k,\quad\n_\mathrm c:=\n\cap \k,
\]
and define two vector spaces
\[
 \q_\mathrm n:=\q/\q_\mathrm c,\quad\n_\mathrm n:=\n/\n_\mathrm c.
\]
Note that $L$ and  $L_\rc$ are connected \cite[Lemma 5.10]{KV}. We introduce three irreducible representations $\tau_V$, $\tau_F$ and $\tau_\n$ of $K$ such that
\be\label{isotau}
  (\tau_V)^{\n_\mathrm c} \cong F^\n|_{L_\rc}\otimes \wedge^{\dim \n_\mathrm n}\n_\mathrm n,\quad (\tau_F)^{\n_\mathrm c} \cong F^\n|_{L_\rc},\quad\textrm{and}\quad (\tau_\n)^{\n_\mathrm c} \cong \wedge^{\dim \n_\mathrm n} \n_\mathrm n
\ee
as representations of $L_\rc$. Then $\tau_V$ occurs with multiplicity one in $V$ (this is the bottom layer of the cohomological induction, and is the unique minimal $K$-type of $V$, in the sense of Vogan); $\tau_F$ occurs with multiplicity one in $F$; and $\tau_\n$ occurs with multiplicity one in both $\wedge^{\dim \n_\mathrm n} \g/\k$ and $\wedge^{\dim \q_\mathrm n} \g/\k$.

\begin{lem}\label{apptau}
One has that
\[
  \dim \Hom_K(\tau_\n, \tau_F^\vee\otimes \tau_V)=1.
\]
\end{lem}
\begin{proof}
Note that $\tau_V$ is the Cartan product of $\tau_F$ and $\tau_\n$. (For details on Cartan products, see \cite{Ea} for example.) Hence
\[
 \dim \Hom_K(\tau_\n, \tau_F^\vee\otimes \tau_V)=\dim \Hom_K(\tau_\n\otimes \tau_F,\tau_V)=1.
\]
\end{proof}

\begin{thml}\label{unitaryc}
(a) For all $j\in \Z$,  $\oH^j(\g, K; F^\vee\otimes V)=\Hom_{K}(\wedge^j \g/\k, F^\vee\otimes V)$.

(b) The space $\oH^j(\g, K; F^\vee\otimes V)$ is zero unless $\dim \n_\mathrm n\leq j\leq \dim \q_\mathrm n$, and both $\oH^{\dim \n_\mathrm n}(\g, K; F^\vee\otimes V)$ and $\oH^{\dim \q_\mathrm n}(\g, K; F^\vee\otimes V)$ are one-dimensional.
\end{thml}
\begin{proof}
Part (a) is proved in \cite[Proposition 9.4.3]{W1}, and part (b) is implied by \cite[Theorem 9.6.6]{W1}.

\end{proof}

Theorem \ref{unitaryc} and Lemma \ref{apptau} imply that we have identifications
\be\label{idenhapp}
  \oH^j(\g, K; F^\vee\otimes V)=\Hom_K(\tau_\n, \tau_F^\vee\otimes \tau_V), \quad \textrm{for $j=\dim \n_\mathrm n$ or $\dim \q_\mathrm n$.}
\ee

\subsection{Non-vanishing of modular symbols at infinity}

Let $H$ be a $\theta$-stable closed subgroup of $G$ with finitely many connected components. Put $H_\rc:=H\cap K$. Let  $\chi_F$ and $\chi_V$ be two characters of $H$ such that
\be\label{chife}
 (\chi_F\cdot \chi_V)|_{(\h,H_\rc)}\cong \wedge^{\dim \h/\h_\rc} \h/\h_\rc,
\ee
where $\wedge^{\dim \h/\h_\rc} \h/\h_\rc$ carries the trivial representation of $\h$ and the adjoint representation of $H_\rc$. Suppose we have two nonzero elements
\[
  \lambda_F\in \Hom_H(F^\vee,\chi_F)\quad\textrm{and}\quad \lambda_V\in \Hom_{\h, H_\rc}(V,\chi_V).
\]
By restriction of cohomology, the functional $\lambda_F\otimes \lambda_V: F^\vee\otimes V\rightarrow \chi_F\otimes \chi_V$ induces a linear map
\be\label{msi}
  \oH^{\dim \h/\h_\rc}(\g,K; F^\vee\otimes V)\rightarrow \oH^{\dim \h/\h_\rc}(\h, H_\rc; \chi_F\otimes \chi_V)\cong \C.
\ee
The functional \eqref{msi} reflects the Archimedean behaviors of various types of modular symbols which are used in the arithmetic study of special values of L-functions. We call it a modular symbol at infinity.  In the literature, authors  concentrate on the cases when $\dim \h/\h_\rc=\dim \n_\mathrm n$ or $\dim \q_\mathrm n$. See \cite{AG, Har, GHL, KMS} for examples. The modular symbol is interesting only when the functional  \eqref{msi} is nonzero.

Recall the parabolic subalgebra $\q$ from Section \ref{unzero}.

\begin{thml}\label{nonvmod}
Assume that $\h+\q=\g$ and $\lambda_V$ does not vanish on the $K$-subrepresentation $\tau_V$ of $V$. If either
\be\label{hintq}
  \dim \h/\h_\rc=\dim \n_\mathrm n\quad \textrm{and}\quad \h\cap \q\subset \k,
\ee
or
\be\label{hintq2}
  \dim \h/\h_\rc=\dim \q_\mathrm n\quad \textrm{and}\quad \h\cap \n\subset \k,
\ee
then the linear functional \eqref{msi} is nonzero.
\end{thml}

We remark that the condition $\h\cap \q=\h\cap \bar \q$ implies that $\h\cap \n=0$ and hence $\h\cap\n\subset \k$. The condition \eqref{hintq} holds, for example, in the case of Rankin-Selberg convolutions for $\GL(n)\times \GL(n-1)$ (see \cite{Sun}).

The rest of this appendix is devoted to a proof of Theorem \ref{nonvmod}.

\begin{lem}\label{liaalgn}
If \eqref{hintq} or \eqref{hintq2} holds, then every nonzero element of the one-dimensional space $\Hom_K(\wedge^{\dim \h/\h_\rc}\g/\k, \tau_\n)$ does not vanish on the one-dimensional subspace $\wedge^{\dim \h/\h_\rc}\h/\h_\rc$ of $\wedge^{\dim \h/\h_\rc}\g/\k$.
\end{lem}
\begin{proof}
We give a proof under the assumption that \eqref{hintq} holds, which is similar to that of \cite[Lemma 2.10]{Sun}. The same proof works when \eqref{hintq2} holds.
Fix a $K$-invariant Hermitian inner product $\la\,,\,\ra$ on $\g/\k$.
 It induces a $K$-invariant  Hermitian inner product $\la\,,\,\ra_\wedge$ on $\wedge^{\dim \h/\h_\rc}\g/\k$.

View $\tau_\n$ as a $K$-submodule of $\wedge^{\dim \h/\h_\rc}\g/\k$. It is generated by the one-dimensional space $\wedge^{\dim \h/\h_\rc}\n/\n_\mathrm c$.
Note that every nonzero element of
\[
  \Hom_K(\wedge^{\dim \h/\h_\rc}\g/\k, \tau_\n)
   \]
   is a scalar multiple of the orthogonal projection $\wedge^{\dim \h/\h_\rc}\g/\k\rightarrow \tau_\n$. Therefore in order the prove the lemma,
it suffices to show that the one-dimensional spaces
$\wedge^{\dim \h/\h_\rc}\h/\h_\rc$ and $\wedge^{\dim \h/\h_\rc}\n/\n_\mathrm c$
are not perpendicular to each other under the form
$\la\,,\,\ra_{\wedge}$. This is equivalent to saying that the pairing
\[
  \la\,,\,\ra: \h/\h_\rc\times \n/\n_\mathrm c\rightarrow \C
\]
is non-degenerate.
Note that
\[
   \{x\in \g/\k\mid \la x, \n/\n_\mathrm c\ra=0\}=\bar \q/\bar \q_\mathrm c,
\]
and by \eqref{hintq}, $\h/\h_\rc\cap \bar \q/\bar \q_\mathrm c=\{0\}$. This proves the lemma.
\end{proof}

\begin{lem}\label{cartanp}
Let $u$ be a nonzero element of $\tau_\n$ and $v$ a nonzero element of $\tau_F$. Then every nonzero element of $\Hom_K(\tau_\n\otimes \tau_F, \tau_V)$ does not vanish on $u\otimes v$.
\end{lem}
\begin{proof}
The lemma holds since $\tau_V$ is the Cartan product of $\tau_\n$ and $\tau_F$ (\cf \cite[Section 2.1]{Ya}).
\end{proof}



\begin{lem}\label{aa6}
If $\h+\q=\g$, then
\be\label{homcl1}
\dim \Hom_{H_\rc}(\tau_V, \chi_V)\leq 1.
\ee
\end{lem}
\begin{proof}
Recall from \eqref{isotau} that $\dim \tau_V^{\n_\mathrm c}=1$.  Note that $\h+\q=\g$ implies $\h_\rc+\q_\mathrm c=\k$. Then we have that
\be\label{taueg}
  \tau_V=\oU(\k). \tau_V^{\n_\mathrm c}=\oU(\h_\rc).(\oU(\q_\mathrm c).\tau_V^{\n_\mathrm c})=\oU(\h_\rc).\tau_V^{\n_\mathrm c}
\ee
Therefore every element of $\Hom_{H_\rc}(\tau_V, \chi_V)$ is determined by its restriction to the one-dimensional space $\tau_V^{\n_\mathrm c}$. Hence \eqref{homcl1} holds.
\end{proof}

\begin{lem}\label{aaa6}
Assume that $\h+\q=\g$ and $\lambda_V$ does not vanish on $\tau_V$. Then $(\chi_V)|_{H_\rc}$ occurs with multiplicity one in $(\tau_V)|_{H_\rc}$, and $\lambda_V$ does not vanish on $(\chi_V)|_{H_\rc}\subset (\tau_V)|_{H_\rc}$.
\end{lem}
\begin{proof}
This is obviously implied by  Lemma \ref{aa6}.
\end{proof}

\begin{lem}\label{fnonv}
View $\tau_F^\vee$ as a $K$-submodule of $F^\vee$. If $\h+\q=\g$, then $\lambda_F$ does not vanish on $\tau_F^\vee$.
\end{lem}
\begin{proof}
Similar to \eqref{taueg}, we have that
\[
  F^\vee=\oU(\h). (F^\vee)^{\bar \n}.
\]
Since $\tau_F^\vee$ is generated by $(F^\vee)^{\bar \n}$, the lemma follows.
\end{proof}

We are now ready to prove Theorem \ref{nonvmod}.  By \eqref{idenhapp}, the map \eqref{msi} is identified with the linear map
\be\label{homkfe1}
  \Hom_K(\tau_\n, \tau_F^\vee\otimes \tau_V)\rightarrow \Hom_C(\wedge^{\dim \h/\h_\rc}\h/\h_\rc,\chi_F\otimes \chi_V)
\ee
which is induce by the linear functional
\[
 (\lambda_F)|_{\tau_F^\vee}\otimes (\lambda_V)|_{\tau_V}:\tau_F^\vee\otimes \tau_V \rightarrow \chi_F\otimes \chi_V
\]
and the restriction to $\wedge^{\dim \h/\h_\rc}\h/\h_\rc$ of the $K$-equivaraint projection map
\[
  p_\n: \wedge^{\dim \h/\h_\rc}\g/\k\rightarrow \tau_\n.
\]
The map \eqref{homkfe1} is identified with the obvious  linear map
\be\label{homkfe2}
  \Hom_K(\tau_\n\otimes \tau_F, \tau_V)\rightarrow \Hom_C(\wedge^{\dim \h/\h_\rc}\h/\h_\rc\otimes \chi_F^\vee, \chi_V).
\ee
The non-vanishing of the map \eqref{homkfe2} is equivalent to saying that the following composition map is nonzero:
\be \label{comw}
  \wedge^{\dim \h/\h_\rc}\h/\h_\rc\otimes \chi_F^\vee\rightarrow \tau_\n\otimes \tau_F\rightarrow \tau_V\stackrel{(\lambda_V)|_{\tau_V}}\longrightarrow \chi_V.
\ee
Here the first arrow is the tensor product of $(p_\n)|_{\wedge^{\dim \h/\h_\rc}\h/\h_\rc}$ and  the transpose of $(\lambda_F)|_{\tau_F^\vee}$, and the second arrow is a nonzero $K$-equivariant linear map. The first arrow is nonzero by Lemma \ref{liaalgn} and Lemma \ref{fnonv}. By Lemma \ref{cartanp}, the composition of the first two arrows of \eqref{comw} is nonzero, and hence by \eqref{chife}, its image is a one-dimensional submodule of $(\tau_V)|_{H_\rc}$ which is isomorphic to $(\chi_V)|_{H_\rc}$. Finally, by Lemma \ref{aaa6}, the composition of \eqref{comw} is nonzero. This finishes the proof of Theorem \ref{nonvmod}.

\end{document}